\documentclass{amsart} 

\newtheorem{theorem}{Theorem}[section]
\newtheorem{lemma}[theorem]{Lemma}
\newtheorem{corollary}[theorem]{Corollary}

\newtheorem{prop}[theorem]{Proposition}
\newtheorem{defn}[theorem]{Definition}
\newtheorem{example}[theorem]{Example}
\newtheorem{remark}[theorem]{Remark}

\newcommand{\Z}{\mathbb{Z}}
\newcommand{\R}{\mathbb{R}}
\newcommand{\C}{\mathbb{C}}

\newcommand{\cV}{\mathcal{V}}

\newcommand{\Gr}{\textrm{Gr}}

\newcommand{\Hom}{\mathrm{Hom}}

\newcommand{\End}{\mathrm{End}}
\newcommand{\Der}{\mathrm{Der}}

\newcommand{\Ker}{\mathrm{Ker}\,}
\renewcommand{\Im}{\mathrm{Im}\,}

\newcommand{\GL}{\mathrm{GL}}

\newcommand{\Span}{\mathrm{Span}}

\newcommand{\tr}{\mathrm{tr} \,}

\newcommand{\bP}{\mathbb{P}}

\newcommand{\cM}{\mathcal{M}}
\newcommand{\cA}{\mathcal{A}}

\newcommand{\cL}{\mathcal{L}}
\newcommand{\pt}{\mathrm{pt}}

\newcommand{\Id}{\mathrm{Id}}

\usepackage[T1]{fontenc} 
\usepackage[utf8]{inputenc} 
\usepackage{graphicx} 
\usepackage{hyperref} 
\usepackage{multirow} 
\usepackage{tabularx} 
\usepackage{color} 
\usepackage{textcomp} 
\usepackage{amsmath} 
\usepackage{amssymb} 
\usepackage{amsfonts} 
\usepackage{amsxtra} 
\usepackage{wasysym} 
\usepackage{isomath} 
\usepackage{mathtools} 
\usepackage{txfonts} 
\usepackage{upgreek} 
\usepackage{enumerate} 
\usepackage{tensor} 
\usepackage{pifont} 
\usepackage{soul} 
\usepackage{arydshln} 
\usepackage[english]{babel}
\usepackage{subcaption}
\usepackage{tikz-cd}

\definecolor{color-1}{rgb}{0.36,0.61,0.84}

\begin{document}

\title[NC geometry in computational models and uniformization for quiver]{Noncommutative Geometry of computational models and Uniformization for framed Quiver Varieties}

\author{George Jeffreys}
\address{Department of Mathematics and Statistics, Boston University, 111 Cummington Mall, Boston MA 02215, USA}
\email{georgej@bu.edu}

\author{Siu-Cheong Lau}
\address{Department of Mathematics and Statistics, Boston University, 111 Cummington Mall, Boston MA 02215, USA}
\email{lau@math.bu.edu}

\begin{abstract}
We formulate a mathematical setup for  computational neural networks using noncommutative algebras and near-rings, in motivation of quantum automata.  We study the moduli space of the corresponding framed quiver representations, and find moduli of Euclidean and non-compact types in light of uniformization.
\end{abstract}

\maketitle

\section{Introduction}

The connections between computer science and algebra are profound. In the early 1900's, both were deeply tied to practical and philosophical developments towards understanding what it truly means to calculate something. For example, there was Turing's Halting problem and G\"odel's Incompleteness theorem. 

As modern abstract algebra was developed in the 50's and 60's, it was fruitfully turned towards this path with the creation of the theory of finite automata. The first fundamental result in this development was Kleene's Theorem demonstrating that the class of recognizable languages is the class of rational languages \cite{Kleene+2016+3+42}. In 1956, Sch\"utzenberger defined the \textit{syntactic monoid}, a canonical monoid attached to each language \cite{Schutzenberger1955-1956}. Later, he proved that a language is star-free exactly when its syntactic monoid is finite and aperiodic \cite{SCHUTZENBERGER1965190}. At this point mathematicians started to consider the algebraic geometry of these monoids as Birkhoff \cite{birkhoff_1935} and latter Eilenberg \cite{eilenberg1976automata} and Reiterman \cite{Reiterman} developed wrote about varieties of these monoids (infinite and finite respectively). Thus we have a well established and important connection between theoretical computer science and pure algebraic geometry.

The theory of finite automata arose from an extremely widespread interdisciplinary effort to understand calculation. 
Modern science suggests that the brain operates as a so-called neural network, the structure of which has inspired the computational tool known as the artificial neural network.  Neural network models heavily use graphs and their linear representations.  This gives rise to further deep relations between mathematics and computer science.  

In this paper, we build an algebraic abstraction that models a neural network and quantum automata.  We are motivated as follows.  A finite automata consists of a set of states of a machine, a set of transitions between the states, and an alphabet set that will form a machine language, whose elements label the transitions of states. 

A quantum version of this replaces the set of states by a collection of vector spaces whose elements are known as state vectors.  The set of transitions is replaced by a set of linear maps between the vector spaces.  This forms a so-called quiver representation, which is a linear representation of the directed graph $Q$ (called a quiver) whose vertices label the collection of vector spaces, and whose arrows label the set of linear maps.

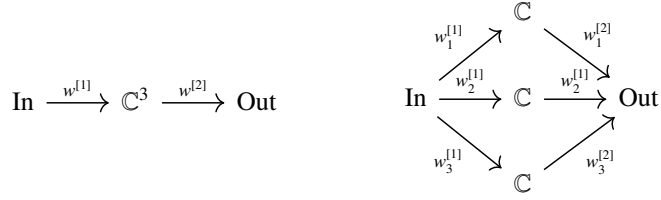
\begin{figure}[h!]
	\begin{subfigure}{0.4\linewidth}
		\centering
		\begin{tikzcd}
			\mathrm{In} \arrow[r, "w^{[1]}"] & \C^3 \arrow[r, "w^{[2]}"] & \mathrm{Out}
		\end{tikzcd}
	\end{subfigure}
	\begin{subfigure}{0.4\linewidth}
		\centering
		\begin{tikzcd}
			& \C \arrow[rd, "w^{[2]}_{1}"] & \\
			\mathrm{In} \arrow[ru, "w^{[1]}_{1}"] \arrow[r, "w^{[1]}_{2}"] \arrow[rd, "w^{[1]}_{3}"'] & \C \arrow[r, "w^{[1]}_{2}"] & \mathrm{Out}\\
			& \C \arrow[ru, "w^{[2]}_{3}"'] &
		\end{tikzcd}
	\end{subfigure}
	\caption{Two artificial neural networks with similar graphs.}
	\label{fig:two-nets}
\end{figure} 

Paths in the quiver play the role of words of a machine language.  The path algebra $$\cA = \C Q$$ 
consists of complex linear combinations of paths, with concatenation of paths serving as the product.  Taking linear combinations can be interpreted as forming superpositions of quantum states.

In summary, \emph{a quiver algebra and its modules provide a nice model of a quantum automata}.

One crucial component that one cannot miss is \emph{taking observation of the quantum particles}.  Most mathematical physics literature concentrate on the quantum propagation process, and have left away the mysterious observation step, perhaps due to its probabilistic and singular nature.  However, this step is crucial in true understanding of quantum physics, and also in practical applications.  For modeling quantum propagations, operator algebras serve as a very successful mathematical tool. However, to include the observation process, we find that a \emph{near-ring}, which is much less studied than an algebra, is necessary.

To model the observation process in a quantum world, we need two more ingredients: Hermitian metrics $h$ of the state spaces $V$, and a framing linear map $e: F \to V$ where $F=\C^{n}$ is called a framing vector space.  Then we take 
$$e^{*_h}(v) = \sum_{j=1}^{n} h(e(\epsilon_j),v) \epsilon_j^*,$$
where $\epsilon_j$ denotes the standard basis of $\C^n$ (and $\epsilon_j^*$ denotes the dual basis).  The coefficients $h(e(\epsilon_j),v)$ are interpreted as the quantum amplitudes of a state $v$ being $e(\epsilon_j)$.  Then the quantum collapsing after observation is modeled by composing this with a fixed non-linear activation function $\sigma:F \to F$ (for instance a certain step function, or a smoothing of it).  In the quantum world, $\sigma$ is indeed an $F$-valued probability distribution on $F$.

Thus, a quantum machine consists of not just linear transitions of states, but also the framings and non-linear activation functions that correspond to taking observations.  We will make the following definition.  See also Figure \ref{fig:computer}.

\begin{defn}[Definition \ref{def:cm}]
	An \emph{activation module} consists of:
	\begin{enumerate}
		\item a (noncommutative) algebra $\cA$ and vector spaces $V,\, F = F_{\mathrm{in}}\oplus F_{\mathrm{out}} \oplus F_{\mathrm{m}}$;
		(`m' stands for `memory' or `middle'.)
		\item A family of metrics $h_{\left(w,e\right)}$ on $V$ over the space of framed $\cA$-modules
		\begin{equation*}
		R = \mathrm{Hom}_{\mathrm{alg}}\left(\cA,\mathrm{End}\left(V\right)\right)\times \mathrm{Hom}\left(F,V\right)
		\end{equation*}
		which is $\mathrm{GL}\left(V\right)$-equivariant;
		\item a collection of possibly non-linear functions
		\begin{equation*}
		\sigma _{j}^{F}\colon F_{\mathrm{m}} \to F_{\mathrm{m}}.
		\end{equation*}		
	\end{enumerate}
\end{defn}

In above, $R$ parametrizes computing machines that have the same underlying framed quiver, and hence is governed by the same language.  Moreover, framed $\cA$-modules that differ by a $\GL(V)$-action have the same computational effect and hence should be identified.  $[R/\GL(V)]$ forms a moduli stack of computing machines.

In this formulation, a machine language is composed of not just linear transitions of state spaces, but also non-linear (or probabilistic) operations $\sigma$ that models quantum observations.  The set of operations generated by these is no longer an algebra, since 
$$\sigma \circ (\gamma_1 + \gamma_2) \not= \sigma \circ \gamma_1 + \sigma \circ \gamma_2$$
where $\gamma_1,\gamma_2$ are composed of linear operations in $\cA$ and the dual framing map $e^{*_h}$.  Rather, it generates a near-ring $\widetilde{A}$, which is almost a ring except that the multiplication (which is realized by composition of maps in the current setup) fails to be distributive on one side.

Motivated by this, we extend the theory of noncommutative differential forms by Connes \cite{Connes}, Cuntz-Quillen \cite{CQ}, Ginzburg \cite{Ginzburg-quiver} to the context of near-rings.  The main idea is that, every element in the near-ring $\widetilde{A}$, which is interpreted as a program written in the language of $\widetilde{A}$, produces a family of maps on the framing space $F$ over the moduli of machines $[R/G]$, that is, each machine in $[R/G]$ performs a computation $F \to F$ specified by the program.  This statement naturally extends to differential forms.

\begin{theorem}[Theorem \ref{thm:Atildeform}]
	There exists a degree-preserving map $$DR^\bullet(\widetilde{\cA}) \to  (\Omega^\bullet(R,\mathbf{Map}\left(F ,F\right)))^{G}$$
	which commutes with $d$ on the two sides.
	In above, $\mathbf{Map}\left(F ,F\right)$ denotes the trivial bundle $\mathrm{Map}\left(F ,F\right) \times R$, and the action of $G=\GL(V)$ on fiber direction is trivial.
\end{theorem}

$0$-forms and $1$-forms are particularly important for machine learning.  Namely, for a fixed algorithm $\tilde{\gamma} \in \widetilde{\cA}$, a learning process attempts to find a machine $p \in [R/G]$ that produces the best fit computation $\phi^{\tilde{\gamma}}_p: F \to F$ by minimizing a certain $0$-form (for instance $\int _{K}\left| \varphi_p^{\tilde{\gamma}}(x)-f\left(x\right)\right|^{2} dx$ for a given $f:K\to\R$ and $K\subset F$ in supervised learning).  Its differential, which is a $1$-form in $DR^1(\widetilde{\cA})$, governs the gradient flow on $[R/G]$ with the help of a metric.

In general, $[R/G]$ is a singular stack.  Fortunately, for quiver algebras, a fine moduli of framed quiver representations was constructed by taking a GIT quotient (with respect to a suitably chosen stability condition) \cite{King,Nakajima-JAMS}.  Such moduli spaces $\cM$ can be used in place of $[R/G]$ and their topologies are well studied by \cite{Reineke}.

In \cite{JL}, we formulated learning of neural networks over the moduli spaces $\cM$.  Namely, the state space $V_i$ over each vertex $i\in Q_0$ patches up as a universal bundle $\cV_i$ over $\cM$.  The transition arrows $a\in Q_1$ correspond to bundle maps over $\cM$.  The framing linear maps $e_i: F_i \to V_i$ correspond to bundle maps from the trivial bundle $\mathbf{F_i}$ to $\cV_i$. Then data and states of the family of machines are naturally modeled by sections over $\cM$; propagation of signals is modeled by bundle maps.  In this formulation, learning is a stochastic gradient descent over the moduli $\cM$.

It is tempting to ask how this formulation relates to the most common method of machine learning over an Euclidean space, rather than a moduli space $\cM$.  In this paper, we will answer this question in light of uniformization of metrics.

The main observation is that, $\cM$ in effect is a compactification of the most commonly used Euclidean space, now denoted as $\cM^0$.  Moreover, the Euclidean space $\cM^0$ can be interpreted as a moduli space of positive-definite quiver representations with respect to a certain Hermitian form $H_i^0$ for the universal bundles $\cV_i$.  Thus, the most popular approach using Euclidean space indeed also falls into our formulation of learning in the moduli space of computing machines.

This uniformization picture naturally includes a hyperbolic version of the moduli space.  Namely, by changing the signature of the quadratic form (see \eqref{eq:H_i^-}), we obtain another type of moduli space $\cM^-$ of positive-definite quiver representations with respect to $H_i^-$.  We show that $\cM^-$ comes with a natural metric.

\begin{theorem}[Theorem \ref{corollary:Kahler}]
	Define $H_T^-$ to be $H_T^-:=-i\sum\limits_{i}\partial\overline{\partial}\log\det H_i^-$ on $\mathcal{M}^-$. Then $H_T^-$ is a K\"ahler metric on $\mathcal{M}^-$.
\end{theorem}

In classical applications, one can also restrict to real coefficients.  Correspondingly, the formulae provided by this paper give bundle metrics for $\cV_i^\R|_{\cM_\R}$ and Riemannian metrics on $\cM_\R$.  

As a result, we can run machine learning over $\cM, \cM^0, \cM^-$, or an interpolation of them.  We can also set learnable parameters that interpolate these spaces, and let the machine learn which metric serves the best for a given task.

\subsection*{Some related works}
Recently, there is a rising interest in the connections between neural networks and quiver representations.  The paper \cite{Armenta-Jodoin} found an interesting way of encoding the data flow as a quiver representation, which makes a crucial use of the assumption of thin representations (where dimensions of representing vector spaces over vertices are all $1$).  On the other hand, the learning that they take is not directly carried over the quiver moduli, and hence is different from our approach in \cite{JL} and this paper.  \cite{GW} studied the symmetries coming from the quiver approach to neural networks.

There are also newly invented approaches to apply higher mathematics to machine learning.  Most literature concerns about the input data set and endows it with more interesting mathematical structures, for instance, Lie group symmetry \cite{CW, CGW, CWKW-spherical, CWKW, CAWHCW, dehaan2020natural}, or categorical structures \cite{Sheshmani-You}.  On the other hand, in our current approach, we focus on the computing machine itself, and formulate its algebro-geometric structure and makes use of its internal symmetry.

For learning using hyperbolic spaces, 
there are several beautiful works, see for instance \cite{Nickel2017PoincarEF}, \cite{Ganea2018HyperbolicEC}, \cite{Sala2018RepresentationTF}, \cite{Ganea2018HyperbolicNN}.  The non-compact dual of the moduli space $\cM^-$ that we introduce in this paper can be understood as a higher rank generalization of hyperbolic spaces in the sense of Hermitian symmetric spaces.  See more in Section \ref{sec:hyp}.

\subsection*{Organization of this paper} In Section 2, we will define computing machines in the context of noncommutative geometry. In Section 3, we will apply the idea of uniformization of metrics to construct non-compact duals to neural network quiver moduli spaces.

\section{An AG formulation of computing machine}

In this section, we give a mathematical formulation of a computing machine based on algebra and geometry.  First, we formulate a machine as a framed module over an algebra, together with a metric on the module and a collection of non-linear functions.  Second, we take into account of isomorphisms of framed modules and make sure the construction is equivariant under the automorphism group, and hence descends to the moduli stack of framed modules.  Finally, we extend the noncommutative geometry developed by \cite{Connes, CQ, Ginzburg-quiver} to the context of near-rings, and show how it fits into this framework.

\subsection{Intuitive construction}
Let $\cA$ be an associative algebra with unit $1_\cA$.  This algebra encodes all possible linear operations of the machine.  Later, in the context of neural network, we will take $\cA$ to be the path algebra of a directed graph (which is also called a quiver).

Let $V$ be a vector space.  $V$ is understood as the space of abstract states of the machine prior to any physical observation.  It is basis-free, namely, we do not pick any preferred choice of basis.  

We consider $\cA$-module structures, that is algebra homomorphisms $w\colon \cA\rightarrow \mathfrak{gl}\left(V\right)$.  Each module structure $w$ realizes $a\in A$ as a linear operation on the state space.

In reality, data are observed and recorded in fixed basis.  For this, we define a framing vector space $F=F_{\mathrm{in}}\oplus F_{\mathrm{out}}\oplus F_{\mathrm{m}}$.  Each component is a vector space with a fixed basis.  We may simply write $F = \C^n$ with the standard basis. Moreover, we consider linear maps $e\colon F\rightarrow V$, $e=e_{\mathrm{in}}\oplus e_{\mathrm{out}}\oplus e_{\mathrm{m}}$ which are called the framing maps.  $F_{\mathrm{in}}\oplus F_{\mathrm{out}}$ are vector spaces of all possible inputs and outputs.  $F_{\mathrm{m}}$ can be understood as a space for memory of the machine.  The framing maps $e$ are used to observe and record the abstract states.

A triple $\left(V,w,e\right)$ is called a framed $\cA$-module.  We denote by
\begin{equation*}
R\coloneqq \left\{\left(w,e\right)\colon w\colon \cA\rightarrow \mathfrak{gl}\left(V\right) \text{ alg. homo.};~ e\colon F\rightarrow V\right\}
\end{equation*}
the \emph{set of framed modules}.  It serves as the parameter space of the machine.  $R$ is a subvariety in $\mathrm{Lin}(\cA,\mathfrak{gl}(V)) \times \mathrm{Lin}(F,V)$.

Let $\cA_{\mathrm{m}}$ be the augmented algebra 
\begin{equation}
\cA_{\mathrm{m}} = \cA\langle 1_{\mathrm{m}}, \mathfrak{e}_{\mathrm{m}}, \mathfrak{e}^*_{\mathrm{m}} \rangle / I
\label{eq:A_m}
\end{equation}  
where $I$ is the two-sided ideal generated by the relations 
\begin{align*}
&1_{\mathrm{m}} \cdot \mathfrak{e}_{\mathrm{m}},\, \mathfrak{e}_{\mathrm{m}} \cdot 1_{\mathrm{m}} - \mathfrak{e}_{\mathrm{m}},\,1_\cA\cdot \mathfrak{e}_{\mathrm{m}} - \mathfrak{e}_{\mathrm{m}},\\ &\mathfrak{e}_{\mathrm{m}}^* \cdot 1_{\mathrm{m}},\, 1_{\mathrm{m}} \cdot \mathfrak{e}_{\mathrm{m}}^* - \mathfrak{e}_{\mathrm{m}}^*,\,
\mathfrak{e}_{\mathrm{m}}^* \cdot 1_\cA - \mathfrak{e}_{\mathrm{m}}^*,\\ &\mathfrak{e}_{\mathrm{m}}^2, \, (\mathfrak{e}^*_{\mathrm{m}})^2, \, a\cdot \mathfrak{e}^*_{\mathrm{m}}, \, \mathfrak{e}_{\mathrm{m}} \cdot a,\, a\cdot 1_{\mathrm{m}}, \, 1_{\mathrm{m}} \cdot a
\end{align*}
for all $a \in \cA$.  (This means, for instance, $1_{\mathrm{m}} \cdot \mathfrak{e}_{\mathrm{m}} = 0$ and $\mathfrak{e}_{\mathrm{m}} \cdot 1_{\mathrm{m}} = \mathfrak{e}_{\mathrm{m}}$ in the algebra $\cA_{\mathrm{m}}$.)  The unit of $\cA_{\mathrm{m}}$ is $1_\cA + 1_{\mathrm{m}}$.

Let's equip $V$ with a Hermitian metric $h$.  Then for each framing map $e=e_{\mathrm{in}}\oplus e_{\mathrm{out}}\oplus e_{\mathrm{m}}$, the element $\mathfrak{e}_{\mathrm{m}} \in \cA_{\mathrm{m}}$ is realized as the map $e_{\mathrm{m}}: F_{\mathrm{m}} \to V$, and $\mathfrak{e}_{\mathrm{m}}^*$ is realized as the metric adjoint $(h(e_{\mathrm{m},l}, \cdot))_{l=1}^{n_{\mathrm{m}}}: V \to F_{\mathrm{m}} = \C^{n_{\mathrm{m}}}$.  

To consider linear maps that have domain and target being $V$, 
we can form the subalgebra $$\cA_{{\mathrm{m}},0} := \cA\cdot \cA_{\mathrm{m}} \cdot \cA.$$  
An element $a\in \cA_{{\mathrm{m}},0}$ is understood as a linear algorithm.  Fixing $(w,e)\in R$, each linear algorithm $a \in \cA_{{\mathrm{m}},0}$ is associated with
$f^{a}\colon F_{\mathrm{in}}\rightarrow F_{\mathrm{out}}$,
\begin{equation*}
f^{a}\left(v\right)\coloneqq e_{\mathrm{out}}^{\mathrm{*}}\left(a\cdot e_{\mathrm{in}}\left(v\right)\right)
\end{equation*}
which is called a machine function.  ($e_{\mathrm{out}}^{*}:V \to F_{\mathrm{out}}$ is the metric adjoint $(h(e_{\mathrm{out},l}, \cdot))_{l=1}^{n_{\mathrm{out}}}$.)
In other words, we have the map
$$ R \times \cA_{{\mathrm{m}},0} \to \mathrm{Hom}(F_{\mathrm{in}},F_{\mathrm{out}}) $$
which is linear in the second component.

So far, this is just a linear model.  In order to capture non-linearity, we also need to incorporate with \emph{non-linear operations} $\sigma _{1},\ldots , \sigma_{N}$.  Let's define these as functions $V\rightarrow V$ for the moment.  (In the next subsection, we shall see that defining in this way is not good from the moduli point of view and it will be modified.)

Consider the $\C$-near-ring $\tilde{\cA}=\cA\left\{\varsigma _{1},\ldots ,\varsigma _{N}\right\}$.  The elements  $\varsigma _{j}$ are algebraic symbols for recording the non-linear operations $\sigma_j$.  See Definition \ref{def:near-ring} for the notion of a near-ring.  Essentially it is recording the compositions of module maps and the non-linear operations.  Similar to above, we take the augmented near-ring
\begin{equation}
\tilde{\cA}_{\mathrm{m}} = \tilde{\cA}\langle 1_{\mathrm{m}}, \mathfrak{e}_{\mathrm{m}}, \mathfrak{e}^*_{\mathrm{m}} \rangle / \tilde{I}
\label{eq:Atildem}
\end{equation}
where $\tilde{I}$ is generated by the relations in $I$ as in \eqref{eq:A_m}, together with the relations
$$\varsigma_l \cdot 1_{\mathrm{m}}, \,\,\, 1_{\mathrm{m}} \cdot \varsigma_l.$$ 
 (This means $\sigma_l$ and $1_{\mathrm{m}}$ compose to be zero.  We want this since $\sigma_l$ is acting on $V$ and $1_{\mathrm{m}}$ is acting on $F_{\mathrm{m}}$.)
An element $\tilde{\gamma}\in \tilde{\cA}_{\mathrm{m},0} := \cA\cdot \tilde{\cA}_{\mathrm{m}}\cdot \cA$ is understood as a non-linear algorithm.

Fixing $\left(w,e\right)\in R$, each algorithm $\tilde{\gamma} \in \tilde{\cA}_{\mathrm{m},0}$ is associated with a non-linear machine function $f_{\left(w,e\right)}^{\tilde{\gamma }}\colon F_{\mathrm{in}}\rightarrow F_{\mathrm{out}}$ ,
\begin{equation}
f_{\left(w,e\right)}^{\tilde{\gamma }}\left(v\right)=e_{\mathrm{out}}^{\mathrm{*}}\left(\tilde{\gamma }\circ _{\left(w,e\right)}e_{\mathrm{in}}\left(v\right)\right).
\label{eq:f^gammatilde}
\end{equation}
That is, we have the map
$$ R \times \tilde{\cA}_{\mathrm{m},0} \to \mathrm{Map}(F_{\mathrm{in}},F_{\mathrm{out}}).  $$

\subsection{Construction over moduli spaces}

An important principle in mathematics and physics is that isomorphic objects should produce the same result.
In other words, we want to have $f_{\left(w,e\right)}^{\tilde{\gamma }}$ well-defined over the moduli stack of framed $\cA$-modules $\mathcal{M} = [R/G]$ for $G=\GL(V)$.  Let's recall the following definition.

\begin{defn}
	For two framed $\cA$-modules $(V, w, e)$ and $(V',w',e')$, where both $e$ and $e'$ have the same domain $F$, a morphism (or an isomorphism) from $(V, w, e)$ to $(V',w',e')$ is a linear map (or a linear isomorphism) $g: V \to V'$ such that $w'(a) \circ g = g \circ w(a)$ for all $a \in \cA$ and $e' = g\circ e$.
\end{defn}

Unfortunately, this is not the case in the above formulation due to the presence of non-linear functions $\sigma: V \to V$.  Any useful non-linear function $\sigma: V \to V$ cannot satisfy $\mathrm{GL}\left(V\right)$-equivariance:
\begin{equation}
g\cdot \left(\sigma \left(v\right)\right)=\sigma \left(g\cdot v\right) \textrm{ for all } g\in \GL(V).
\label{eq:equiv-sigma-naive}
\end{equation}
It produces a crucial gap between the subject of machine learning and representation theory.

Here is a simple solution to this problem.  Let $\cV$ be the universal bundle over the moduli stack $\cM$, which is descended from the trivial bundle $V \times R$, where $G=\GL(V)$ acts diagonally.  

Rather than defining $\sigma$ as a single linear map $V\rightarrow V$, let's take $\sigma$ to be a fiber-bundle map $V \times R \to V \times R$ over $R$.  Then $\sigma$ descends as a fiber-bundle map $\cV \to \cV$ over $\cM$ if it satisfies 
the equivariance equation
\begin{equation}
g\cdot \left(\sigma _{\left(w,e\right)}\left(v\right)\right)=\sigma _{\left(g\cdot w,g\cdot e\right)}\left(g\cdot v\right) \textrm{ for all } g\in \GL(V).
\label{eq:equiv-sigma}
\end{equation}
The difference between Equation \eqref{eq:equiv-sigma} and \eqref{eq:equiv-sigma-naive} is that $\sigma$ is now allowed to also depend on $(w,e)\in R$.   

Now suppose we have $\GL(V)$-equivariant fiber-bundle maps $\sigma_1,\ldots,\sigma_N: V\times R \to V \times R$.  As in the last subsection, we have the map $ R \times \tilde{\cA}_{\mathrm{m},0} \to \mathrm{Map}(F_{\mathrm{in}},F_{\mathrm{out}})$ by realizing $\varsigma_i \in \tilde{\cA}$ as $(\sigma_i)_{(w,e)}:\cV\to \cV$.

Recall that we have used a Hermitian metric on $V$ for taking the adjoint of framing $e^*$.  To make sure $e^*$ is also equivariant, we need to equip $V$ with a family of Hermitian metrics $h_{\left(w,e\right)}$ for $\left(w,e\right)\in R$, in a $\GL(V)$-equivariant way:
\begin{equation}
	h_{\left(g\cdot w,g\cdot e\right)}\left(g\cdot u,g\cdot v\right)=h_{\left(w,e\right)}\left(u,v\right)  \textrm{ for all } g\in \GL(V).
	\label{eq:h-equiv}
\end{equation}
That is, $h$ descends to be a Hermitian metric on the universal bundle $\cV$ over $\cM$.
Note that we are NOT asking for $\GL(V)$-invariance
$h\left(g\cdot u,g\cdot v\right)=h\left(u,v\right)$ for a single metric $h$, which is impossible.  

\begin{prop} \label{prop:equiv}
	In the above setting, the non-linear machine function defined by Equation \eqref{eq:f^gammatilde} satisfies the equivariance $f_{\left(w,e\right)}^{\tilde{\gamma }}=f_{g\cdot\left(w,e\right)}^{\tilde{\gamma }}$  for all $g\in \GL(V)$.  
\end{prop}

\begin{proof}
	The fiber-bundle map $f_{\left(w,e\right)}^{\tilde{\gamma }}: V \times R \to V \times R$ defined by \eqref{eq:f^gammatilde} is a composition of $e_{\mathrm{out}}^{\mathrm{*}}=(h_{(w,e)}(e_{\mathrm{out},l}, \cdot))_{l=1}^{n_{\mathrm{out}}}$, $w_a$ for $a \in \cA$, the fiber-bundle maps $(\sigma_i)_{(w,e)}:V\times R \to V \times R$, and $e_{\mathrm{in}}$.  Under the action of $g \in \GL(V)$, They change to $$e_{\mathrm{out}}^{\mathrm{*}}=(h_{g\cdot (w,e)}(g\cdot e_{\mathrm{out},l}, \cdot))_{l=1}^{n_{\mathrm{out}}}=(h_{(w,e)}(e_{\mathrm{out},l}, g^{-1}(\cdot)))_{l=1}^{n_{\mathrm{out}}}=e_{\mathrm{out}}^{\mathrm{*}}\cdot g^{-1},$$
	$g\cdot w_a \cdot g^{-1}$, $$\sigma _{\left(g\cdot w,g\cdot e\right)} = g \cdot \sigma_{(w,e)}(g^{-1}(\cdot))$$ and $g \cdot e_{\mathrm{in}}$ respectively, using Equation \eqref{eq:equiv-sigma} and \eqref{eq:h-equiv}.  The composition remains the same.
\end{proof}

In this way, we obtain the map
$\cM \times \tilde{\cA}_{\mathrm{m},0} \to \mathrm{Map}(F_{\mathrm{in}},F_{\mathrm{out}})$.

In applications, we need concrete fiber bundle maps $\sigma: \cV \to \cV$.  They can be cooked up using the Hermitian metric $h$ on $\cV$ as follows.  Given any function $\sigma ^{F}\colon F_{\textrm{m}}\rightarrow F_{\textrm{m}}$, define $\sigma _{\left(w,e\right)}$ as
\begin{equation*}
\sigma _{\left(w,e\right)}\left(v\right)\coloneqq e^{(\textrm{m})}\cdot \sigma ^{F}\left(h_{\left(w,e\right)}\left(e^{(\textrm{m})}_{1},v\right),\ldots ,h_{\left(w,e\right)}\left(e^{(\textrm{m})}_{n_{\textrm{m}}},v\right)\right).
\end{equation*}
In other words, we observe and record the state $v$ to memory using $e^{(\textrm{m})}$ and $h$; then we perform the non-linear operation $\sigma^F$ on the memory $F_{\textrm{m}}$; finally we send it back as a state in $V$.  Unlike the setting in the last subsection, the non-linear operation $\sigma^F$ is now defined on the framing space $F_{\textrm{m}}$ instead of on the basis-free state space $V$.

\begin{prop} \label{prop:equiv-sigma}
	The above $\sigma _{\left(w,e\right)}\colon V\times R\rightarrow V\times R$ is $\mathrm{GL}\left(V\right)$-equivariant.
\end{prop}

\begin{proof}
	\begin{align*}
		\sigma _{\left(g\cdot w,g\cdot e\right)}\left(g\cdot v\right) &=  g\cdot e^{(\textrm{m})}\cdot \sigma ^{F}\left(h_{\left(g\cdot w,g\cdot e\right)}\left(g\cdot e^{(\textrm{m})}_{1}, g\cdot v\right),\ldots ,h_{\left(g\cdot w,g\cdot e\right)}\left(g\cdot e^{(\textrm{m})}_{n_{\textrm{m}}},g\cdot v\right)\right) \\
		&= g \cdot e^{(\textrm{m})}\cdot \sigma ^{F}\left(h_{\left(w,e\right)}\left(e^{(\textrm{m})}_{1},v\right),\ldots ,h_{\left(w,e\right)}\left(e^{(\textrm{m})}_{n_{\textrm{m}}},v\right)\right) = g\cdot \sigma_{(w,e)}(v)
	\end{align*}
	using Equation \eqref{eq:h-equiv}.
\end{proof}

The non-linear operations are called activation functions in machine learning.  We conclude the current setting by the following definition.  

\begin{defn} \label{def:cm}
	An \emph{activation module} consists of:
	\begin{enumerate}
		\item a (noncommutative) algebra $\cA$ and vector spaces $V,\, F = F_{\mathrm{in}}\oplus F_{\mathrm{out}} \oplus F_{\mathrm{m}}$;
		\item A family of metrics $h_{\left(w,e\right)}$ on $V$ over the space of framed $\cA$-modules
		\begin{equation*}
		R = \mathrm{Hom}_{\mathrm{alg}}\left(\cA,\mathrm{End}\left(V\right)\right)\times \mathrm{Hom}\left(F,V\right)
		\end{equation*}
		which is $\mathrm{GL}\left(V\right)$-equivariant;
		\item a collection of possibly non-linear functions
		\begin{equation*}
		\sigma _{j}^{F}\colon F_{\mathrm{m}} \to F_{\mathrm{m}}.
		\end{equation*}		
	\end{enumerate}
	The data of (1) and (2) (without (3)) is called a Hermitian family of framed modules.
\end{defn}

Figure \ref{fig:computer} shows a schematic picture of an activation module.

\begin{figure}[h]
	\centering
	\includegraphics[scale=0.4]{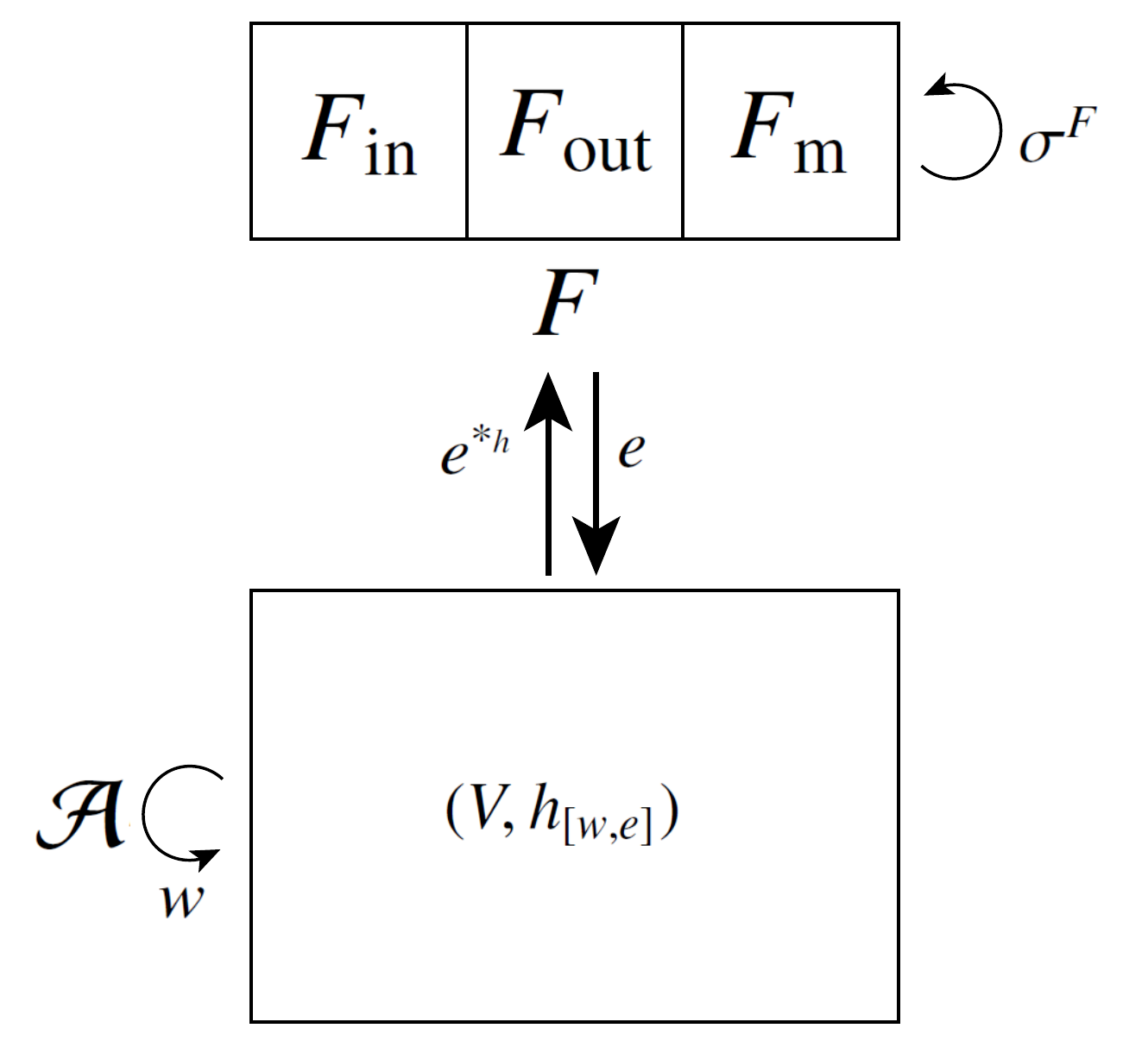}
	\caption{}
	\label{fig:computer}
\end{figure}

In this setting, $\sigma_j^F$ is a function on $F_{\mathrm{m}}$.  We take the subalgebra
\begin{equation}
\cL(\cA_{\mathrm{m}}) := \mathfrak{e}_{\mathrm{m}}^* \cdot  \cA_{\mathrm{m}} \cdot \mathfrak{e}_{\mathrm{m}}
\label{eq:LAm}
\end{equation}
consisting of loops at $F_{\mathrm{m}}$, the near-ring
\begin{equation}
\widetilde{\cA_{\mathrm{m}}}:=(\cL(\cA_{\mathrm{m}}))\{\varsigma_1,\ldots,\varsigma_N\},
\label{eq:Amtilde}
\end{equation}
and
$$ (\widetilde{\cA_{\mathrm{m}}})_0 := \cA \cdot \mathfrak{e}_{\mathrm{m}} \cdot \widetilde{\cA_{\mathrm{m}}} \cdot \mathfrak{e}_{\mathrm{m}}^* \cdot \cA. $$
Note that $\widetilde{\cA_{\mathrm{m}}}$ is different from $\tilde{\cA}_{\mathrm{m}}$ in Equation \eqref{eq:Atildem}, since we now have non-linear functions defined on $F$ instead of $V$.  

Using Proposition \ref{prop:equiv} and \ref{prop:equiv-sigma}, each algorithm $\tilde{\gamma} \in (\widetilde{\cA_{\mathrm{m}}})_0$ and $[w,e] \in \cM$ gives a machine function $f^{\tilde{\gamma}}_{[w,e]}$.  This gives a map
$$(\widetilde{\cA_{\mathrm{m}}})_0 \to \Gamma \left(\cM,\mathbf{Map}\left(F_{\mathrm{in}},F_{\mathrm{out}}\right)\right).$$

In applications, an activation module may consist of several linear submodules, which are connected by possibly-nonlinear transitions $\sigma^F_i$.  This means the algebra $\cA$ is a direct sum $\bigoplus_{k\in K} \cA^{(k)}$ where each $\cA^{(k)}$ is understood as a linear component of the activation module (and $K$ is an index set).  Similarly we have $V = \bigoplus_{k\in K} V^{(k)}$ and $F = \bigoplus_{k\in K} F^{(k)}$.  We take the moduli stack $\prod_{k\in K} [R^{(k)}/\GL(V^{(k)})]$ (where $R^{(k)}=\mathrm{Hom}_{\mathrm{alg}}\left(A^{\left(k\right)},\mathrm{End}\left(V^{\left(k\right)}\right)\right)\times \mathrm{Hom}\left(F^{\left(k\right)},V^{\left(k\right)}\right)$) instead of $[R/\GL(V)]$.  Each $F^{(k)}$ has three components $F^{(k)} = F^{(k)}_{\mathrm{in}}\oplus F^{(k)}_{\mathrm{out}} \oplus F^{(k)}_{\mathrm{m}}$ (where some of the components can simply be $\{0\}$).  
Furthermore, the non-linear functions $\sigma^F_j:F_{\mathrm{m}} \to F_{\mathrm{m}}$ is a composition $\iota\circ s_j^{F} \circ \pi$, where $s_j^{F}: F_{\mathrm{m}}^{\left(p_{j,1}\right)}\times \ldots \times F_{\mathrm{m}}^{\left(p_{j,{m_{j}}}\right)}\rightarrow F_\mathrm{m}^{\left(q_{j,1}\right)}\times \ldots \times F_\mathrm{m}^{\left(q_{j,{n_{j}}}\right)}$ for some fixed $\{p_{j,1},\ldots,p_{j,{m_{j}}}\}$ and $\{q_{j,1},\ldots,q_{j,{n_{j}}}\}$; $\pi$ is the projection $F_\mathrm{m} \to F_{\mathrm{m}}^{\left(p_{j,1}\right)}\oplus \ldots \oplus F_{\mathrm{m}}^{\left(p_{j,{m_{j}}}\right)}$ and $\iota$ is the inclusion (or extension by zero) $F_\mathrm{m}^{\left(q_{j,1}\right)}\oplus  \ldots \oplus  F_\mathrm{m}^{\left(q_{j,{n_{j}}}\right)} \to F_\mathrm{m}$.  Finally, $h$ is a direct sum $h_{\left(w,e\right)}=\bigoplus_{k\in K} h_{\left(w^{\left(k\right)},e^{\left(k\right)}\right)}$ where each $h_{\left(w^{\left(k\right)},e^{\left(k\right)}\right)}$ is a family of metrics $h_{\left(w^{\left(k\right)},e^{\left(k\right)}\right)}$ on $V^{\left(k\right)}$ over the space of framed $A^{\left(k\right)}$-modules
$
R^{(k)}
$
which is $\mathrm{GL}\left(V^{\left(k\right)}\right)$-equivariant.

We can also define a closely related setting that uses unitary framed modules, which takes the unitary group $U(V,h)$ in place of $\GL(V)$, and takes a single Hermitian metric $h$ in place of a family of Hermitian metrics.

\begin{defn}
	A \emph{unitary activation  module} consists of:
	\begin{enumerate}
		\item A Hermitian vector space $(V,h)$, a framing vector space $F = F_{\mathrm{in}}\oplus F_{\mathrm{out}} \oplus F_{\mathrm{m}} = \C^n$ (equipped with the standard metric), and unitary framing maps $e_\bullet: F_\bullet \to V$, where $\bullet = \mathrm{in},\mathrm{out},\mathrm{m}$.
		\item A group ring $\cA = \C[G]$ where $G$ is a subgroup of the unitary group $U(V,h)$.  $\C[G]$ consists of linear combinations $\sum_{g\in G} c_g g$ for $c_g \in \C$.  
		\item a collection of possibly non-linear functions
		\begin{equation*}
		\sigma _{j}^{F}\colon F_{\mathrm{m}} \to F_{\mathrm{m}}.
		\end{equation*}	
		\end{enumerate}
\end{defn}

Such a setting suits well for quantum computing.  Namely, $(V,h)$ can be taken to be the state space of a quantum system of particles.  $G$ is a subgroup of unitary operators on $(V,h)$.  $F_{\textrm{m}}$ can be taken to have the same dimension as $V$, and $e_{\textrm{m}}: F_{\textrm{m}} \to V$ maps the standard basis of $F_{\textrm{m}}$ to an assigned unitary basis of $V$.  (For instance, the assigned basis can be $\{|00\rangle, |01\rangle, |10\rangle, |11\rangle\}$ for a 2-qubit system).  There is a probabilistic projection $\sigma_0:F_{\textrm{m}} \to F_{\textrm{m}}$ that corresponds to wave-function collapse following each observation.  We also have other non-linear classical operations $\sigma_j^F$ on $F_{\textrm{m}}$.
	
In application, we are given  input data $v \in F_{\textrm{in}}$.  $v$ (normalized to have length $1$) is sent to the Hermitian state space $V$ by $e_{\textrm{in}}$, and operated under a prescribed linear algorithm $a \in \C[G]$.   Then the system is observed and recorded using the basis $e_{\textrm{m}}$.  This gives $\sigma_0 \cdot \sum_l h(e_{\textrm{m},l},a\cdot e_{\textrm{in}} \cdot v) e_{\textrm{m},l}$.  The recorded memory can be operated by a non-linear algorithm consisting of $\sigma_j^F$.  The process can be iterated and give a function $F_{\textrm{in}} \to F_{\textrm{out}}$.

In this paper, we focus on Definition \ref{def:cm}, for the purpose of neural network and deep learning which works with $\GL(n)$ rather than $U(n)$.

\subsection{Noncommutative geometry and machine learning}
\label{sec:ncml}

We have formulated a computing machine by a Hermitian family of framed $\cA$-modules and a collection of non-linear functions.  If we ignore the non-linear functions for the moment, and merely consider the augmented algebra $\cA_{\mathrm{m}}$, it fits well to the framework of noncommutative geometry developed by Connes \cite{Connes}, Cuntz-Quillen \cite{CQ}, Ginzburg \cite{Ginzburg-quiver}.  Below we give a quick review and apply to our situation.  \cite{Tacchella} gives a beautiful survey on this theory.  We will extend it to near-ring in the next subsection.

\subsubsection{A quick review}
The theory develops an analog of the de Rham complex of differential forms for an associative algebra $A$ over a field $\mathbb{K}$ (that we take to be $\C$ in this paper).  This is a crucial step to develop the notions of cohomology, connection and curvature for the noncommutative space associated to $A$ and its associated vector bundles.

The noncommutative differential forms can be described as follows.  Consider the quotient vector space $\overline{A} = A/\mathbb{K}$ (which is no longer an algebra).  We think of elements in $\overline{A}$ as differentials.  Define
$$ D(A) := \bigoplus_{n\in \Z_{\geq 0}} D(A)_n, \,\, D(A)_n := A\otimes \overline{A} \otimes \ldots \otimes \overline{A} $$
where $n$ copies of $\overline{A}$ appear in $D(A)_n$, and the tensor product is over the ground field $\mathbb{K}$.  We should think of elements in $\overline{A}$ as \emph{matrix-valued} differential one-forms.  Note that $X\wedge X$ may not be zero, and $X \wedge Y \not= -Y \wedge X$ in general for matrix-valued differential forms $X,Y$.

The differential $d_n: D(A)_n \to D(A)_{n+1}$ is defined as
$$
d_n(a_0\otimes \overline{a_1} \otimes \ldots \otimes \overline{a_n}) := 1\otimes \overline{a_0} \otimes \ldots \otimes \overline{a_n}.  $$
The product $D(A)_n \otimes D(A)_{m-1-n} \to D(A)_{m-1}$ is more tricky:

\begin{align}
&(a_0\otimes \overline{a_1} \otimes \ldots \otimes \overline{a_n})\cdot (a_{n+1}\otimes \overline{a_{n+2}} \otimes \ldots \otimes \overline{a_m}) \nonumber \\
:=& (-1)^n a_0a_1 \otimes \overline{a_2} \otimes \ldots \otimes \overline{a_m} + \sum_{i=1}^n (-1)^{n-i} a_0\otimes \overline{a_1} \otimes \ldots \otimes \overline{a_ia_{i+1}} \otimes \ldots \otimes \overline{a_m}
\label{eq:prod}
\end{align}

which can be understood by applying the Leibniz rule on the terms $\overline{a_ia_{i+1}}$.
Note that we have chosen representatives $a_i \in A$ for $i=1,\ldots,n+1$ on the RHS, but the sum is independent of choice of representatives (while the product $\overline{a_ia_{i+1}}$ itself depends on representatives).  

The above product in particular gives a bimodule structure on $D(A)$ over $A = D(A)_0$.  For instance, $D(A)_1$ has the bimodule structure
$$ a\cdot (a_0 \otimes \overline{a_1}) = aa_0 \otimes \overline{a_1},\,\, (a_0 \otimes \overline{a_1})\cdot a = -a_0a_1\otimes \overline{a} + a_0\otimes \overline{a_1 a}. $$
(If $a_1$ is replaced by $a_1+k$ for $k\in \mathbb{K}$, then RHS $= -a_0a_1\otimes \overline{a} - k a_0\otimes \overline{a}  + a_0\otimes \overline{a_1 a} + k a_0 \otimes \overline{a} = -a_0a_1\otimes \overline{a} + a_0\otimes \overline{a_1 a}$ remains unchanged.)

By \cite{CQ}, 
$$ d^2=0. $$
The above differential $d$ and product defines a dg-algebra structure on $D(A)$; indeed this is the unique one that satisfies $a_0 \cdot da_1 \cdot \ldots \cdot da_n = a_0 \otimes \overline{a_1}\otimes\ldots\otimes \overline{a_n}$. Moreover, $(D(A),i)$, where $i: A \to D(A)_0=A$ is the identity map, has the following universal property: for every $(\Gamma,\psi)$ where $\Gamma$ is a dg algebra and $\psi: A \to \Gamma_0$ is an algebra homomorphism, there exists an extension as a dg-algebra map $u_\psi: D(A) \to \Gamma$ such that the degree-zero part satisfies $(u_\psi)_0 \circ i = \psi$.

Here is another realization of differential forms for $A$.  First, define the $A$-bimodule $\Omega^1(A) := \Ker(\mu)$ where $\mu: A\otimes A \to A$ is the multiplication map for $A$.  Moreover, define $d: A \to \Omega^1(A)$ by $da := 1\otimes a - a\otimes 1$.  Thus $\sum_i a_i da_i'$ for $a_i,a_i' \in A$ is an element in $\Omega^1(A)$.  Conversely, any element in $\Omega^1(A)$ is of the form $\sum_i a_i \otimes a_i'$ with $\sum_i a_i \cdot a_i' = 0$, and this equals to
$$ \sum_i a_i d a_i'  = -\sum_i (da_i) a_i'.$$

Then we take the tensor algebra
$$\Omega^\bullet(A) := T_A(\Omega^1(A)) = \bigoplus_{i \in \Z_{\geq 0}} \Omega^1(A) \otimes_A \ldots \otimes_A \Omega^1(A)$$
where there are $i$ copies of $\Omega^1(A)$ for the summands on the right.  An element in $\Omega^\bullet(A)$ takes the form
$ a_1 db_1 \otimes_A a_2 db_2 \otimes_A \ldots \otimes_A a_k db_k \cdot a_{k+1}$.  Recall that tensoring over $A$ means the identification $db_1 \cdot a \otimes_A db_2 = db_1 \otimes_A a db_2$.

The two defined graded algebras $\Omega^\bullet(A)$ and $D(A)$ are isomorphic.  For one forms, we have the $A$-bimodule map $\psi: \Omega^1(A) \to D(A)_1$ defined by $da \mapsto 1\otimes \overline{a}$.  It has the inverse $a_0 \otimes \overline{a_1} \mapsto a_0 \otimes a_1 - a_0a_1 \otimes 1$ (which is again independent of choice of representative $a_1$).  For higher forms, $\Omega^n \to D(A)_n$ is given by $\alpha_1 \otimes_A \ldots \otimes_A \alpha_n \mapsto \psi(\alpha_1)\cdot  \ldots \cdot \psi(\alpha_n)$ (where the non-trivial product on $D(A)$ is given in Equation \eqref{eq:prod}), whose inverse is $a_0\otimes \overline{a_1} \otimes \ldots \otimes \overline{a_n} = (a_0 \otimes \overline{a_1})\cdot (1 \otimes \overline{a_2})\ldots (1 \otimes \overline{a_n}) \mapsto \psi^{-1}(a_0\otimes \overline{a_1}) \otimes_A \psi^{-1}(1\otimes \overline{a_2}) \otimes_A \ldots \otimes_A \psi^{-1}(1\otimes \overline{a_n})$.

The Karoubi-de Rham complex is defined as
\begin{equation}
	DR^\bullet (A) := \Omega^\bullet (A) / [\Omega^\bullet (A),\Omega^\bullet (A)]
	\label{eq:DR}
\end{equation}
where $[a,b] := ab - (-1)^{ij} ba$ is the graded commutator for a graded algebra.  $d$ descends to be a well-defined differential on $DR^\bullet (A)$.  Note that $DR^\bullet (A)$ is not an algebra since $[\Omega^\bullet (A),\Omega^\bullet (A)]$ is not an ideal.  $DR^\bullet (A)$ is the non-commutative analog for the space of de Rham forms.  Moreover, there is a natural map by taking trace to the space of $G$-invariant differential forms on the space of representations $R(A)$:
\begin{equation}
DR^{\bullet}\left(A\right)\rightarrow \Omega ^{\bullet}\left(R\left(A\right)\right)^{G}.  \label{eq:DR-map}
\end{equation}

$DR^0(A)$ and $DR^1(A)$ will be the most relevant to us.  We have $DR^0(A) = A/[A,A]$ and $DR^1(A) = \Omega^1(A)/[A,\Omega^1(A)]$. 

Dually, derivations $\theta \in \Der (A)$ play the role of vector fields.  A derivation $\delta: A \to A$ is a linear map satisfying $\delta (ab) = \delta(a)\cdot b + a\cdot \delta(b)$.  $\Der (A)$ is the vector space of all derivations.  We have the $A$-bimodule map $\iota_\theta: \Omega^1(A)\to A$, $\iota_\theta(da) := \theta(a)$ called contraction.  $\iota_\theta$ extends to $\Omega^\bullet (A) \to \Omega^{\bullet-1} (A)$ by using graded Leibniz rule, and descends to $DR^\bullet(A) \to DR^{\bullet-1}(A)$.

The following version of differential forms relative to a subalgebra \cite{CQ} will be useful for framings and quivers.  Let $B \subset A$ be a commutative subalgebra.  We take
$$ D(A/B)_n := A \otimes_B \bar{A} \otimes_B \ldots \otimes_B \bar{A} $$
where $\bar{A}$ is the vector space
$$\bar{A} := A/B. $$
Then we repeat the same definitions as above for $DR^\bullet (A/B)$.  Note that zero-th forms are the same as before: $DR^0\bullet (A/B) = DR^0\bullet (A)$.  There is a natural map \cite{Ginzburg-quiver}
$$ DR^\bullet (A/B) \to \Omega ^{\bullet}\left(R_B\left(A\right)\right)^{G_B}$$
where $R_B(A)$ is the set of $A$-modules whose restriction to $B$ equals to a prescribed $B$-module, and $G_B$ is the subgroup in $\GL(V)$ that preserves the prescribed $B$-bimodule structure.

In the context of $A$ being the path algebra of a quiver, we shall take $B$ to be the subalgebra generated by the trivial paths $1_i$ at all vertices $i\in Q_0$.  Then a differential form 
$$a_0 (da_1) a_2 \ldots (da_k) \in DR^\bullet (A/B)$$ 
is non-zero only if the paths $a_i$ can be concatenated: $t(a_j) = h(a_{j+1})$ for all $j \in \Z/(k+1)$.  In this case a prescribed $B$-module structure on $V$ is given by a decomposition $V = \bigoplus_{i\in Q_0} V_i$ and $1_i$ acts as the projection $V \to V_i$.  Then $G_B = \prod_{i\in Q_0} \GL(V_i)$.

\subsubsection{Application to linear machine learning} 
Now we come back to the context of the last subsection.  The additional ingredient we need to take care of is the equivariant family of Hermitian metrics $h$ on the $A$-modules.

To precisely match the language, first let's modify the definition for $\cA_{\mathrm{m}}$ (Equation \eqref{eq:A_m}) as follows.  Recall that the framing vector space $F = F_{\mathrm{in}}\oplus F_{\mathrm{out}} \oplus F_{\mathrm{m}} = \C^{n_{\mathrm{in}}}\oplus \C^{n_{\mathrm{out}}} \oplus \C^{n_{\mathrm{m}}}$, where $\dim F = n$.  Then a framing $e$ can be written as $(e_{1} \ldots e_{n})$ where $e_j \in V$, and $e^*$ is the column vector $(e^*_{1}, \ldots ,e^*_{n})$ where $e^*_j \in V^*$.  

First we take the augmentation $$\cA^\mathfrak{e} := \cA\langle 1_{F}, \mathfrak{e}_{j}: j=1,\ldots, n \rangle / I $$
where $I$ is the two-sided ideal generated by $1_{F} \cdot \mathfrak{e}_j,\,\mathfrak{e}_{j} \cdot 1_{F} - \mathfrak{e}_{j},\,1_\cA \cdot \mathfrak{e}_{j} - \mathfrak{e}_{j}, \, \mathfrak{e}_{j}\mathfrak{e}_{k},\, \, \mathfrak{e}_{j} \cdot a, 
a\cdot 1_{F}, \, 1_{F}\cdot a$ for all $a \in \cA, j,k =1,\ldots, n$. 

Then we take its \emph{doubling} $\hat{\cA}$, which is generated by two copies of $\cA^\mathfrak{e}$ (whose generators are denoted by $a,1_{F}, \mathfrak{e}_{j}$ and $a^*,1_{F}^*, \mathfrak{e}_{j}^*$ respectively), quotient out the ideal of relations $1_\cA - 1_\cA^*, 1_F - 1_F^*$.  The unit of $\hat{\cA}$ is
$$1_{\hat{\cA}} = 1_F + 1_\cA.$$
We also use the rule $(ab)^*:=b^*a^*$ to define the formal adjoint of a general element in $\hat{\cA}$.

\begin{remark}
	This doubling procedure is standard in the construction of Nakajima quiver varieties, which is an algebraic analog of taking the cotangent bundle (or complexification) of a variety.  We will restrict to a section to go back to $[R/G]$.
\end{remark}


In the notation of the last subsection, we take $A = \hat{\cA}$ and the commutative subalgebra 
$$B = \Span_\C \{1_F,1_\cA\} \subset \hat{\cA}.$$  
Consider $V \oplus \C$.  We fix its $B$-module structure in the way that $1_\cA$ and $1_F$ act as $(\Id_V,0)$ and $(0,\Id_\C)$ respectively.  $V \oplus \C$ can be equipped with $\hat{\cA}$-module structure that restricts to be this fixed $B$-module structure.

\begin{lemma} \label{lem:rep-Ahat}
	Given a Hermitian family of framed modules $(\cA, V, F, h)$, there is a one-to-one correspondence between elements in $R = \mathrm{Hom}_{\mathrm{alg}}\left(\cA,\mathrm{End}\left(V\right)\right)\times \mathrm{Hom}\left(F,V\right)$ and $\hat{\cA}$-modules of the form $V\oplus \C$ that respect the $B$-module structure and have $\mathfrak{e}_{j}^*, a^*$ acting as the adjoints of $e_j$ and $w(a)$ respectively with respect to $h$.
\end{lemma}

\begin{proof}
Given $(w,e) \in R$, the $\hat{\cA}$-module structure on $V\oplus \C$ is defined as follows.  $w$ gives the action of $\cA$ on $V$, and $\cA$ acts on the component $\C$ by zero.  $\mathfrak{e}_{j}$ acts as the linear map $e_j: \C \to V$ where $e_j$ is the $j$-th column of $e$, and acts on $V$ trivially.  
$\mathfrak{e}_{j}^*$ and $a^*$ act on the component $\C$ by zero, and act as the adjoint maps of $\mathfrak{e}_j$ and $w(a)$ with respect to $h$.  The adjoint maps are
$$e_j^{*_h}: V \to \C,  \, e_j^{*_h}(v) = h(e_j,v)$$
and
$$
w(a)^{*_h} = h_{(w,e)}^{-1} w(a)^* h_{(w,e)}
$$
in matrix form.

Conversely, since the $\hat{A}$-module is required to restrict as the given $B$-module structure, we must have $\cA$ acting trivially on the component $\C$, $\mathfrak{e}_{j}$ acting trivially on $V$, and $\mathfrak{e}_{j}^*$ acting trivially on $\C$.  (For instance, $a = a \cdot 1_\cA$ acts as $(a,0)$ on $V\oplus \C$.)  Then the action of $\cA$ and $(\mathfrak{e}_{j}: j=1,\ldots,n)$ gives an element in $R$.
\end{proof}

Similar to \eqref{eq:DR-map}, we have the following map for $\hat{\cA}$.  The only difference is that for the forms $d \mathfrak{e}_{j}^*$ and $da^*$, the corresponding forms on $$R=\mathrm{Hom}_{\mathrm{alg}}\left(\cA,\mathrm{End}\left(V\right)\right)\times \mathrm{Hom}\left(F,V\right)$$ are defined using the metrics $h$.

\begin{prop} \label{prop:lin-ind-form}
	Given a Hermitian family of framed modules $(\cA, V, F, h)$,
	there is a (degree-preserving) map
	$$ DR^\bullet (\hat{\cA}/B) \to \Omega ^{\bullet}\left(R\right)^{G_B}$$
	that commutes with differential, and equals to the trace of the corresponding representations given in Lemma \ref{lem:rep-Ahat} when restricted to $DR^0 (\hat{\cA}/B) \to \Omega ^0\left(R\right)^{G_B}$.
\end{prop}

\begin{proof}
	$DR^\bullet (\hat{\cA}/B)$ is generated by the one forms $da$, $da^*$, $d\mathfrak{e}_{j}$ and $d\mathfrak{e}_{j}^*$ over $\hat{\cA}$.  For  $da$ and $d\mathfrak{e}_{j}$, the corresponding matrix-valued one-forms on $R$ are obvious (by substituting $a$ and $\mathfrak{e}_{j}$ by the corresponding representing matrices $w(a)$ and $e_j$).  For $d\mathfrak{e}_{j}^*$ and $da^*$, the corresponding matrix-valued one-form over $R$ are 
	$$(\bar{\partial} e_j^*) \cdot h + e_j^* \cdot dh = (\bar{\partial} e_j^*) \cdot h + e_j^* \cdot (\bar{\partial} h + \partial h)$$
	and
	\begin{equation} \label{eq: da}
	-h^{-1}\cdot dh \cdot h^{-1} w_{a}^*h + h^{-1} (\bar{\partial} w_{a}^*) h + h^{-1} w_{a}^* dh
	\end{equation}
	respectively,
	where $h$ is now represented by a square matrix in a basis of $V$, $e_j^*$ (a row vector) and $w_a^*$ are the conjugate transpose of $e_j$ and $w_a$ respectively.  Note that $h_{(w,e)}$ is a function on $(w,e) \in R$ and so it has a non-trivial differential $dh$.
	More intrinsically, $d\mathfrak{e}_{j}^*$ corresponds to $h(\nabla e_j, \cdot) + h(e_j, \nabla \cdot)$, where $\nabla$ is the Chern connection of $h$ on the trivial vector bundle $V\times R$ (and $e_j$ is a section).
	
	Note that non-zero elements in $DR^\bullet (\hat{\cA}/B)$ are represented by loops (meaning that the source and target are the same), due to the defining equation \eqref{eq:DR}.
	The corresponding forms on $R$ are obtained by composing the above matrices and taking trace.  In particular, it is the trace of the corresponding representing matrix when restricted to $DR^0(\hat{A}/B)$.  Since trace is independent of cyclic permutations of the composition, the map $DR^\bullet(\hat{A}/B) \to \Omega^\bullet(R)$ is well-defined.   Moreover, it is obvious that it commutes with differential by definition.  
	
	Under the action of $g\in \GL(V)$, $d(w(a)) \mapsto g \cdot d(w(a)) \cdot g^{-1}$, $de_j \mapsto g \cdot de_j$, 
	$$(\bar{\partial} e_j^*) \cdot h + e_j^* \cdot dh \mapsto (\bar{\partial} e_j^*) g^* \cdot (g^*)^{-1}hg^{-1} + e_j^* g^* \cdot (g^*)^{-1}dh\,g^{-1} = ((\bar{\partial} e_j^*) \cdot h + e_j^* \cdot dh) \cdot g^{-1}  $$
	and \eqref{eq: da} transforms by $g\,(\cdot)\, g^{-1}$,
	using the $\GL(V)$-equivariance of the family of metrics $h$.
	Since trace is invariant under conjugation, the corresponding forms on $R$ are $G_B$-invariant.  Here $G_B = \C^\times \times \GL(V)$, where $\C^\times$ is Abelian and acts trivially on $R$.  
\end{proof}

\begin{remark}
	Since the above uses the family of Hermitian metrics $h$, the resulting forms in $\Omega^\bullet(R)^{G_B}$ are no longer holomorphic.  In the usual algebraic construction, we have a map $\rho$ from $DR^p (\hat{\cA}/B)$ to $\GL(V)$-invariant holomorphic $(p,0)$-forms on 
	$$(\mathrm{Hom}_{\mathrm{alg}}\left(\cA,\mathrm{End}\left(V\right)\right))^2 \times \mathrm{Hom}\left(F,V\right)\times \mathrm{Hom}\left(V,F\right).$$
	
	The above can be understood as a composition of the usual map 
	$$\rho: DR^\bullet (\hat{\cA}/B) \to \Omega^\bullet (R\times (\mathrm{Hom}_{\mathrm{alg}}\left(\cA,\mathrm{End}\left(V\right)\right) \times \mathrm{Hom}\left(V,F\right))) $$
	together with pulling back by the smooth section of $R\times (\mathrm{Hom}_{\mathrm{alg}}\left(\cA,\mathrm{End}\left(V\right)\right) \times \mathrm{Hom}\left(V,F\right)) \to R$ defined by 
	$$e_j' = h_{(w,e)}(e_j, \cdot) = e_j^* \cdot h_{(w,e)},\,\, w_a' = h_{(w,e)}^{-1} w_a^* h_{(w,e)}.$$ 
	On the LHS, $(e_j': j=1,\ldots,n) \in \mathrm{Hom}\left(V,F\right)$ and $w_a' \in \End(V)$ denotes fiber coordinates; on the RHS, $e_j^*$ is the conjugate transpose of the column vector in $(e_1\ldots e_n) \in \Hom(F,V)$.  Note that the action of $\GL(V)$ on both sides of the first and second equations are right multiplication by $g^{-1}$ and conjugation $g\,(\cdot)\,g^{-1}$ respectively.
\end{remark}

Now define the subalgebra
$$\cL(\hat{\cA}) := \bigoplus_{j,k=1}^{n} \mathfrak{e}_{j}^* \cdot \hat{\cA} \cdot \mathfrak{e}_{k}.$$
Recall that elements in $\cL(\hat{\cA})$ are understood as linear algorithms.

In $DR^0(A/B)=A/(B+[A,A])$ (vector-space quotient), note that elements that do not form loop (for instance, $a\cdot \mathfrak{e}_{j}$ and $\mathfrak{e}_{j}^* \cdot a$) are in the zero class.  Moreover, loops that are cyclic permutation of each other are identified as the same class.  

In our context,
elements in $\cL(\hat{\cA})$ are loops, and descend to non-trivial elements in $DR^0(A/B)$.  As a consequence:

\begin{corollary} \label{cor:induce-fcn}
	An element in $\cL(\hat{\cA})$ induces a $G$-invariant function $f$ on $R$ where $G=\GL(V)$.  Its differential lies in $DR^1(A/B)$ and induces the corresponding differential $df \in \Omega^1(R)^{G}$.
\end{corollary}

Note that the target of $\mathfrak{e}_{j}^*$ and the domain of $\mathfrak{e}_{j}$ are the one-dimensional vector space $\C$.  Thus the matrix corresponding to $\mathfrak{e}_{j}^* \cdot a \cdot \mathfrak{e}_{k} \in \cL(\hat{\cA})$ is one-by-one whose trace just equals to itself.

An $(n\times n)$-matrix whose entries lie in $\cL(\hat{\cA})$ gives a linear function 
$F \to F$
over each point in $[R/G]$.  We can also restrict it to 
$$f_{[w,e]}: F_{\mathrm{in}} \to F_{\mathrm{out}} $$
by taking an $(n_{\mathrm{out}}\times n_{\mathrm{in}})$-matrix whose entries $\gamma_{jk}$ belong to $\mathfrak{e}_{\mathrm{out},k}^* \cdot \hat{\cA} \cdot \mathfrak{e}_{\mathrm{in},j}$ where $(\mathfrak{e}_{\mathrm{in},j}: j=1,\ldots,n_{\mathrm{in}})$ denotes the part of $(\mathfrak{e}_{j}:j=1,\ldots,n)$ that has source in $F_{\mathrm{in}}$ (and similar for $\mathfrak{e}_{\mathrm{out},k}$).
This produces a linear machine function $f^\gamma_{[w,e]}$ corresponding to a linear algorithm $\gamma$.

The cost function can also be defined algebraically as an element in $DR^0(A/B)$.  Namely, given a function $f: F_{\mathrm{in}} \to F_{\mathrm{out}}$ and fixing $v \in F_{\mathrm{in}} = \C^{n_{\mathrm{in}}}$, the expression
\begin{align*}
E &= \int_{K} \left|\left(\sum_j \gamma_{jk}\, v_j: k=1,\ldots,n_{\mathrm{out}}\right)
-f\left(v\right)\right|_{F_{\mathrm{out}}}^{2} dv \\
&= \int_{K} \sum_k \left(\sum_j \gamma_{jk}\, v_j - f_k(v)\right)\left(\sum_j \gamma_{jk}^*\, \overline{v_j} - \overline{f_k(v)}\right) dv
\end{align*}
lies in $DR^0(A/B)$.  Its differential in $DR^1(A/B)$ induces a one-form on $[R/G]$, which plays a central role in machine learning.

Suppose $\cA$ is finitely generated, and so does $\hat{\cA}$.  Let $\{x_j: j=1,\ldots,M\}$ be the generators of $\hat{\cA}$.  Then the algebraic Jacobian ring
$$ DR^0(\hat{\cA}/B) / \langle \partial_{x_j} E : j=1,\ldots,M \rangle, $$
where $\partial_{x_j} E$ is the cyclic differential, is useful in capturing the critical locus of $E$.

\subsection{Differential forms for near-ring}

The associative algebra $A$ in the last subsection captures linear operations of a computing machine, and has interesting noncommutative geometries.  In this subsection, we incorporate non-linear operations and extend the geometric construction to a near-ring.

\subsubsection{Near-rings and their representations}

\begin{defn} \label{def:near-ring}
	A near-ring is a set $\tilde{A}$ with two binary operations $+, \circ$ called addition and multiplication such that 
	\begin{enumerate}
		\item  $\tilde{A}$ is a group under addition.
		\item Multiplication is associative.
		\item Right multiplication is distributive over addition:
		$$ (x+y)\circ z = x\circ z + y\circ z$$
		for all $x,y,z \in \tilde{A}$.
	\end{enumerate}
\end{defn}

In this paper, the near-ring we use will be required to satisfy that:
\begin{enumerate}
	\item[(4)] $(\tilde{A},+)$ is a vector space over $\mathbb{F}=\C$, with $c\cdot (x\circ y) = (c \cdot x) \circ y$ for all $c \in \C$ and $x,y \in \tilde{A}$.
	\item[(5)] There exists $1 \in \tilde{A}$ such that $1\circ x = x = x \circ 1$.
\end{enumerate}
We call it a near-ring over $\C$ with identity, or a $\C$-near-ring with identity.

Note that $x\circ (c\cdot y) \not= c\cdot x\circ y$ in general.  The following gives a prototype example.

\begin{example}
	The set $\mathrm{Map}(V,V)$ of $\C$-valued smooth functions $f: V \to V$ on a vector space $V$ forms a near-ring over $\C$ with identity, with $+$ being the addition on the vector space, $\circ$ being the composition of functions, and $1$ being the identity function on $V$.
\end{example}
 
 \begin{defn}
 	Given a $\C$-near-ring with identity $\tilde{A}$, a $\C$-sub-near-ring is a $\C$-subspace $\tilde{A}' \subset \tilde{A}$ which is closed under the multiplication $\circ$.  $\tilde{A}'$ is called a $\C$-sub-near-ring with identity if in addition, $1 \in \tilde{A}'$.
 \end{defn}

Given an algebra $A$ and a set $S$, we have the $\C$-near-ring $A\{S\}$ defined as follows.

\begin{defn} \label{def:A-S}
	Let $A$ be a $\C$-algebra with identity and $S$ be a set.  we define the $\C$-near-ring with identity $A\{S\}$ as follows.  As a vector space,
	$$ A\{S\} := \bigoplus_{p=0}^\infty A\{S\}_p $$
	where:
	\begin{enumerate}
		\item $A\{S\}_0 = A$;
		\item Given $A\{S\}_p$ defined, $A\{S\}_{p+1}$ is spanned by the elements $a \varsigma \circ \alpha$, where $a \in A$, $\varsigma \in S$, and $\alpha \in A\{S\}_p$, subject to the relation $(a_1 \varsigma_1 + c a_2 \varsigma_2) \circ  \alpha = a_1 \varsigma_1 \circ \alpha + c a_2 \varsigma_2 \circ \alpha$ for all $c\in \C$, $a_1,a_2 \in A$.
	\end{enumerate}
	Moreover, we define $1_A \circ \varsigma = \varsigma \circ 1_A = \varsigma$.  Thus $1_A$ is also the identity for $A\{S\}$.
\end{defn}

In the application to neural network, the elements $\varsigma \in S$ are symbols for the activation functions.  
Each element of $A\{S\}$ can be recorded by a rooted tree (oriented towards the root) defined as follows.

\begin{defn}
	Given $\tilde{A}=A\{S\}$, an activation tree is a rooted tree with the following labels.
	
	\begin{enumerate}
		\item Leaves and the root are labeled by $1_{\tilde{A}}$;
		\item Edges are labeled by $a\in A$ ;
		\item Nodes that are neither leaves nor the root are labeled by $\varsigma \in S$.
	\end{enumerate}
	Each node gives the output \begin{equation}
		\sum_k a_k \varsigma_k\circ \alpha_k,
		\label{eq:preact}
	\end{equation} 
	where $a_k$ are the labels of the incoming edges, $\varsigma_k$ and $\alpha_k$ are the labels of the tails of the incoming edges and their outputs respectively.  (At a leaf, the label is $1_{\tilde{A}}$ and the output is $1_{\tilde{A}}$.) The element in $A\{S\}$ corresponding to the tree is the output of its root.
\end{defn}

\begin{remark}
	The expression \eqref{eq:preact} takes the pre-activation value as output of a node.  One can also slightly modify the definition of an activation tree and use the other convention that takes the activation value as output.  
\end{remark}

\begin{example}
	Figure \ref{fig:function-tree} shows examples of activation trees that represent elements in $A\{S\}$.  The expression corresponding to the rightmost tree is
	$$ a_0 + a_1 \varsigma_1 \circ (a_{1,0} + a_{1,1}\varsigma_{1,1}\circ a_{1,1,0}) + a_2 \varsigma_2\circ a_{2,0}$$
	for some $a_0,a_1,a_{1,0},a_{1,1},a_{1,1,0},a_{2,0} \in A$, $\varsigma_1,\varsigma_{1,1},\varsigma_2 \in S$.
	
	Note that the tree here is not the digraph (quiver) that we will consider in the later part of this paper.  The labels $a$ for the edges will be taken to be elements in the double of a quiver algebra $\hat{\cA}$ later, and required to be loops from the framing of the quiver back to itself.
\end{example}

\begin{figure}[h]
	\centering
	\includegraphics[scale=0.5]{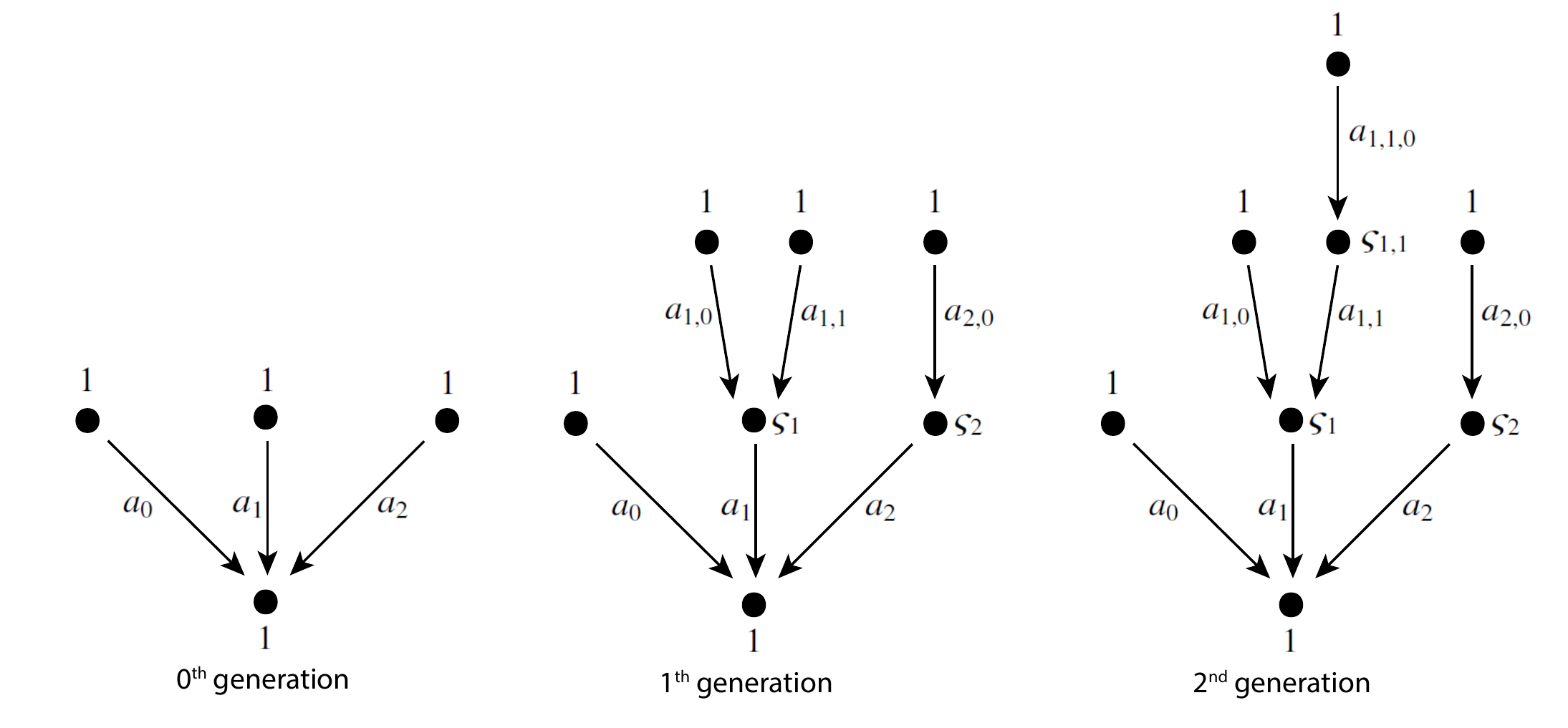}
	\caption{}
	\label{fig:function-tree}
\end{figure}

The above definition goes from a $\C$-algebra to a $\C$-near-ring.  In the reverse direction, we can define the following.

\begin{defn}
The canonical subalgebra of a $\C$-near-ring $\tilde{A}$ with identity is defined as
$$ A := \{ x \in \tilde{A}: x\circ (cy + z) = c x\circ y + x\circ z \textrm{ for all } y,z \in \tilde{A} \textrm{ and } c \in \C \}.$$
\end{defn}

It is easy to check that
\begin{lemma}
	$A$ is a $\C$-algebra with identity.
\end{lemma}

\begin{example}
	For the above example that $\tilde{A}=\mathrm{Map}(V,V)$, the canonical subalgebra is the subset $\mathrm{Lin}(V)$ of linear endomorphisms of $V$.  This can be seen by taking $y,z\in \mathrm{Map}(V,V)$ to be constant maps in the above definition of $A$.
\end{example}

Given a subset $S$ of $\tilde{A}$, we have the sub-near-ring generated by $S$ defined as follows.

\begin{defn}
	The sub-near-ring of $\tilde{A}$ generated by $S$, which is denoted as $\langle S \rangle_{\tilde{A}}$, is defined inductively as follows.  As a vector space,
	$$ \langle S \rangle_{\tilde{A}} := \sum_{p=0}^\infty \langle S \rangle_{\tilde{A},p} \subset \tilde{A}$$
	where:
	\begin{enumerate}
		\item $\langle S \rangle_{\tilde{A},0} = A$;
		\item Given $\langle S \rangle_{\tilde{A},p}$ defined, $\langle S \rangle_{\tilde{A},p+1}$ is spanned by the elements $a \circ \varsigma \circ \alpha$, where $a \in A$, $\varsigma \in S$, and $\alpha \in \langle S \rangle_{\tilde{A},p}$.
	\end{enumerate}

	$\tilde{A}$ is said to be finitely generated if $\tilde{A} = \langle S \rangle_{\tilde{A}}$ for a finite subset $S \subset \tilde{A}$.  $S\subset \tilde{A}$ is said to be a free generating subset if $\langle S \rangle_{\tilde{A}} = A\{S\}$.
\end{defn}

It is easy to check that:
\begin{prop}
	$\langle S \rangle_{\tilde{A}}$ defined above is a sub-near-ring.
\end{prop}

\begin{example}
	Let's continue the example of the set of functions $\mathrm{Map}(V,V)$.  Fix a collection of non-linear functions $\sigma_1,\ldots,\sigma_N:V\to V$.  This corresponds to a finitely generated sub-near-ring $A\{\sigma_1,\ldots,\sigma_N\} \subset \tilde{A}$. $\sigma_1,\ldots,\sigma_N$ can be chosen such that they are not related by iterated compositions and linear combinations.  Then they form a free generating subset.
\end{example}

\begin{defn}
	A morphism of $\C$-near-rings with identities is a map $\Psi: \tilde{A}_1 \to \tilde{A}_2$ that satisfies:
	\begin{enumerate}
		\item $\Psi(x+y)=\Psi(x)+\Psi(y)$;
		\item 
		$\Psi(x\circ y) = \Psi(x) \circ \Psi(y)$;
		\item $\Psi(1_{\tilde{A}_1})=1_{\tilde{A}_2}$.
	\end{enumerate}
	$\Psi$ is said to be a strong morphism if in addition, it satisfies:
	\begin{enumerate}
		\item[(4)] $\Psi$ maps the canonical subalgebra of $\tilde{A}_1$ to that of $\tilde{A}_2$.
	\end{enumerate}
\end{defn}

It easily follows from the definition that a surjective morphism of $\C$-near-rings is automatically strong.

Now we consider modules of a $\C$-near ring.

\begin{defn}
	For a $\C$-near ring $\tilde{A}$ with identity, an $\tilde{A}$-module is a $\C$-vector space $V$ together with a strong  $\C$-near-ring morphism $\tilde{A}\to \mathrm{Map}(V,V)$.
	
	For two $\tilde{A}$-modules $V,W$, a morphism from $V$ to $W$ is a map $\phi\in \mathrm{Map}(V,W)$ that commutes with the actions of $\tilde{A}$:
	$$ \phi \circ V(\alpha)(v) = W(\alpha)\circ \phi(v)$$
	for all $\alpha \in \tilde{A}$.
\end{defn}

It follows from the above definition that an $\tilde{A}$-module is automatically an $A$-module (where $A$ denotes the canonical subalgebra).  

Essentially, the method of deep learning is performing a (stochastic) gradient descent on a certain subvariety of the space of $\tilde{A}$-modules for a fixed near-ring $\tilde{A}$.  However, such a space of $\tilde{A}$-modules is typically infinite-dimensional (since the choice of non-linear maps is infinite-dimensional).  
It is important to systematically construct explicit $\tilde{A}$-modules.  A useful construction for $\tilde{A}=A\{\varsigma_1,\ldots,\varsigma_N\}$ is the following.  Given an algebra and an $A$-module $V$, a choice of $\sigma_1,\ldots,\sigma_N\in \mathrm{Map}(V,V)$ enhances $V$ to be an $A\{\varsigma_1,\ldots,\varsigma_N\}$-module.  (Here, $\varsigma_l$ are the formal symbols corresponding to $\sigma_l$.)

Unfortunately, such a correspondence between $A$-modules and $\tilde{A}$-modules does not behave well in the morphism level.  Namely, an $A$-module endomorphism $\phi \in \mathrm{Lin}(V,V)$ typically does not satisfy
$ \phi \circ \sigma_l = \sigma_l \circ \phi $
for non-linear functions 
$\sigma_l \in \mathrm{Map}(V,V)$, and hence cannot be lifted as an $\tilde{A}$-module morphism.  So we do not have a map from the space of $A$-modules to the space of $\tilde{A}$-modules that descend to isomorphism classes.

Below, we use our setting of an activation module to remedy this correspondence between $A$ and $\tilde{A}$.  See Proposition \ref{prop:A-Atilde}.

\subsubsection{Forms over near-ring}

Let $\cA$ be an algebra, and fix a framing vector space $F=\C^n$.
In Section \ref{sec:ncml}, we have taken the doubled augmented algebra $\hat{\cA}$.  Now, we consider the set $\mathrm{Mat}_F(\hat{\cA})$ of $n\times n$ matrices whose $(k,j)$-th entries lie in $\mathfrak{e}_{k}^* \cdot \hat{\cA} \cdot \mathfrak{e}_{j}$.

It is easy to check that:
\begin{lemma}
	$A:=\mathrm{Mat}_F(\hat{\cA})$ forms an algebra under matrix addition and multiplication (where multiplication between entries is given by $\hat{\cA}$).
\end{lemma}

This is essentially the algebra $\cL(\cA)$ defined in Equation \eqref{eq:LAm}, adapted to the current setting by identifying $\mathfrak{e} = (\mathfrak{e}_{j}:j=1,\ldots,n)$.
As explained previously right after Corollary \ref{cor:induce-fcn}, each element of $\mathrm{Mat}_F(\hat{\cA})$ induces a section of the trivial bundle $\End(F)$ over $[R/G]$, where $R$ is the space of framed representations of $\cA$.

Similar to \eqref{eq:Amtilde}, we take the $\C$-near-ring
$$ \widetilde{\cA} := \mathrm{Mat}_{F}(\hat{\cA})\{\varsigma_1,\ldots,\varsigma_N\}$$
where each $\varsigma_l$ represents a non-linear function $\sigma_l: F \to F$.  

As in Definition \ref{def:A-S}, we have a natural grading on $\widetilde{\cA}$.  Recall that the elements of $\widetilde{\cA}$ can be recorded by rooted trees.  The generation of rooted trees gives a grading on $\widetilde{\cA}$: $$\widetilde{\cA} = \bigoplus_k \widetilde{\cA}_k.$$  
$\widetilde{\cA}_0 = \mathrm{Mat}_{F}(\hat{\cA})$; $\widetilde{\cA}_p$ consists of linear combinations of $a \cdot \varsigma_j \circ \alpha$ for $a\in \mathrm{Mat}_{F}(\hat{\cA})$, $\alpha \in \widetilde{\cA}_{p-1}$, and $j=1,\ldots,N$.

In the last subsection,  we have explained a correspondence between $A$-modules and $A\{\varsigma_{1},\ldots,\varsigma_{N}\}$-modules, by choosing maps $\sigma_1,\ldots,\sigma_N \in \textrm{Map}(V,V)$.  However, such a correspondence does not descend to isomorphism classes.
The advantage of the construction here (after fixing a framing vector space $F$) is that the correspondence is well-defined on the moduli space.

\begin{prop} \label{prop:A-Atilde}
	Fix $\sigma^F_l \in \mathrm{Map}(F,F)$ for $l=1,\ldots,N$.
	A framed $\cA$-module $(V,F)$ with a Hermitian metric $h$ on $V$ induces an $\widetilde{\cA}$-module structure on $F$.  Moreover, if two such modules with metrics are isomorphic $(V,F,h) \cong (V',F',h')$, then the induced $\widetilde{\cA}$-module structures on $F$ are the same.	
	Thus, fixing an equivariant family of metrics on $V$, we have the map 
	$$[R(A)/G] \to R_F\left(\widetilde{\cA}\right)$$
	where $R_F(\widetilde{\cA})$ denotes the space of $\widetilde{\cA}$-module structures on $F$.
\end{prop}

\begin{proof}
	As explained below Corollary \ref{cor:induce-fcn}, by using the framed $\cA$-module structure and metric, each element in $\mathrm{Mat}_{F}(\hat{\cA})$ induces a linear endomorphism of $F$, which is invariant under $\GL(V)$.  Thus two isomorphic framed modules with metrics produce the same linear endomorphism of $F$.  Moreover, $\sigma^F_l$ are maps on $F$ which receive no action by $\GL(V)$.  As a result, this gives an $\tilde{A}$-module structure on $F$ which remains the same for isomorphic $(V,F,h)$.  
\end{proof}

The above proposition explains why we want Definition \ref{def:cm} for an activation module.

\begin{remark}
	In Definition \ref{def:cm}, we have a splitting $F=F_{\mathrm{m}}\oplus F_{\mathrm{in}} \oplus F_{\mathrm{out}}$.  It is easy to restrict to the component $F_{\mathrm{m}}$ (or other components).  We have the projection $p: F\to F_{\mathrm{m}}$ and inclusion $\iota: F_{\mathrm{m}} \to F$.  The functions $\sigma_j^F: F_{\mathrm{m}} \to F_{\mathrm{m}}$ can also be understood as functions on $F$.  From now on, we will simply work with the whole framing vector space $F$, keeping in mind that we can restrict to the components if we want.
\end{remark}

We are going to define differential forms on $\widetilde{\cA}$.  Under the setting of Definition \ref{def:cm}, they will induce $\mathrm{Map}\left(F ,F\right)$-valued forms on $[R/G]$ (Theorem \ref{thm:Atildeform}).  

First, recall that we have the Karoubi-de Rham complex $DR^\bullet(\hat{\cA}/B)$.  It contains the subspace of forms over loops at the framing vertex.  These forms are linear combinations of elements $\mathfrak{e}_{k}^*\ldots \mathfrak{e}_{j}, (d\mathfrak{e}_{k}^*)\ldots \mathfrak{e}_{j}, \mathfrak{e}_{k}^*\ldots (d\mathfrak{e}_{j}), (d\mathfrak{e}_{k}^*)\ldots (d\mathfrak{e}_{j})$ for some $j,k=1,\ldots,n$.  In other words, the subspace is $\sum_{j,k=1}^{n} DR^\bullet(\hat{\cA}/B)_{j,k}$, where $DR^\bullet(\hat{\cA}/B)_{j,k}$ is defined as
$$ \mathfrak{e}_{k}^* \cdot DR^\bullet(\hat{\cA}/B) \cdot \mathfrak{e}_{j} + d\mathfrak{e}_{k}^* \cdot DR^\bullet(\hat{\cA}/B) \cdot \mathfrak{e}_{j} + \mathfrak{e}_{k}^* \cdot DR^\bullet(\hat{\cA}/B) \cdot d\mathfrak{e}_{j} + d\mathfrak{e}_{k}^* \cdot DR^\bullet(\hat{\cA}/B) \cdot d\mathfrak{e}_{j}.$$

We define the linear part as follows.
\begin{defn}
	$DR^\bullet(\mathrm{Mat}_{F}(\hat{\cA}))$ is defined to be the space of $n\times n$ matrices whose $(k,j)$-th entries lie in $DR^\bullet(\hat{\cA}/B)_{j,k}$.  
\end{defn} 
Like $DR^\bullet(\hat{\cA}/B)$, this space is graded by the degree of forms.  

From Proposition \ref{prop:lin-ind-form}, we have the map
\begin{equation}
DR^\bullet(\mathrm{Mat}_{F}(\hat{\cA})) \to (\Omega^\bullet(R,\End\left(F\right)))^{G}.
\label{eq:lin-ind-form}
\end{equation}
($F$, and hence $\End(F)$, are treated as a trivial bundle over $[R/G]$.)

To define differential forms on $\widetilde{\cA}$,  
we need to use the symbols $\left. D^{(p)} \varsigma_l \right|_{\alpha}(a_1,\ldots,a_p)$, which represent the $p$-th order \emph{symmetric} differentials of the non-linear functions $\sigma_l$.  For instance, $D^{(1)} \varsigma_l$ represents the usual differential $d \sigma_l$; $D^{(2)} \varsigma_l$ represents the Hessian of $\sigma_l$, which is a symmetric bilinear two-form.  $D^{(p)} \varsigma_l$ is supersymmetric about its $p$ inputs:
\begin{equation}
\left. D^{(p)} \varsigma_l \right|_{\alpha}(a_1,\ldots,a_k,a_{k+1},\ldots,a_p) = (-1)^{\deg a_k \cdot \deg a_{k+1}}\left. D^{(p)} \varsigma_l \right|_{\alpha}(a_1,\ldots,a_{k+1}, a_k, \ldots, a_p)
\label{eq:Dss}
\end{equation}

where $\deg a$ denotes the degree of $a$.
The inputs $a_i$ are again differential forms on $\widetilde{\cA}$.  The point of evaluation $\alpha$ is an element of $\widetilde{\cA}$.

\begin{defn}
	A form-valued tree is a rooted tree (oriented towards the root) whose edges are labeled by $\phi \in DR^\bullet(\mathrm{Mat}_{F}(\hat{\cA}))$; leaves are labeled by $\alpha\in \widetilde{\cA}$; the root (if not being a leaf) is labeled by $1$; nodes which are neither leaves nor the root are labeled by $\left. D^{(p)} \varsigma_l \right|_{\alpha}$ for some $l=1,\ldots,N$, $\alpha \in \widetilde{\cA}$, and $p>0$ is the number of incoming edges.  
\end{defn}

The trivial rooted tree, which has a single node with no edge, corresponds to zero-form.  The node is attached with an element $\alpha \in \widetilde{\cA}$.

For a non-trivial rooted tree, the output of each node which are neither leaves nor the root is
$$  \left. D^{(p)} \varsigma_l \right|_{\alpha} (\phi_1\cdot \eta_1,\ldots,\phi_p\cdot \eta_p) $$
where $\phi_k\in DR^\bullet(\mathrm{Mat}_{F}(\hat{\cA}))$ are attached to the incoming edges, and $\eta_k$ are the outputs of the nodes adjacent to the incoming edges.  The input edges to the node are read clockwisely.  Its degree is defined as the sum of $\deg (\phi_k \cdot \eta_k) = \deg \phi_k + \deg \eta_k$.  The output of each leaf is simply its label $\alpha\in \widetilde{\cA}$ which has degree 0.  The output of the root, which is the sum of $\phi_k\cdot \eta_k$ for the incoming edges $\phi_k$ and outputs of incoming nodes $\eta_k$, is taken to be the differential form associated to the form-valued tree.

\begin{remark}
	Now we have introduced two different kinds of rooted trees.  The activation tree represents an element in $\widetilde{\cA}$ (which is identified as a zero-form); the form-valued tree represents a $p$-form.  For $p=0$, the form-valued tree is trivial consisting of a single root, which is labeled by $\alpha \in \widetilde{\cA}$.  $\alpha$ is represented by an activation tree, which is more useful in this situation.
\end{remark}

\begin{defn}
	A differential zero-form over $\widetilde{\cA}$ is simply an element in $\widetilde{\cA}$.  Denote
	$$DR^0(\widetilde{\cA}) := \widetilde{\cA}.$$
	
	A differential $p$-form (for $p \geq 1$) is a sum of forms associated to form-valued trees with at most $p$ leaves, with total of degrees of forms attached to edges being $p$.  The space of $p$-forms is denoted by  $DR^p(\widetilde{\cA})$.
\end{defn}

\begin{remark}
	Since we require the trees contributing to a $p$-form to have at most $p$ leaves, $D^{(k)} \varsigma_l$ that appear at the nodes must have $k \leq p$.
\end{remark}

\begin{example}
	Figure \ref{fig:form-tree} shows examples of one-form and two-form.  The corresponding expressions are $a_1 da_2 \cdot D^{(1)}\left.\varsigma_l\right|_{\alpha_1} (a_3 \cdot \alpha_2)$, $a_1 da_2 \cdot D^{(1)}\left.\varsigma_l\right|_{\alpha_1} ((a_3 da_4)\cdot \alpha_2)$ and $a_1 da_2 \cdot D^{(2)}\left.\varsigma_l\right|_{\alpha_1}(a_3 \cdot \alpha_2, (a_4 da_5) \cdot \alpha_3)$ respectively.
\end{example}

\begin{figure}[h]
	\centering
	\includegraphics[scale=0.7]{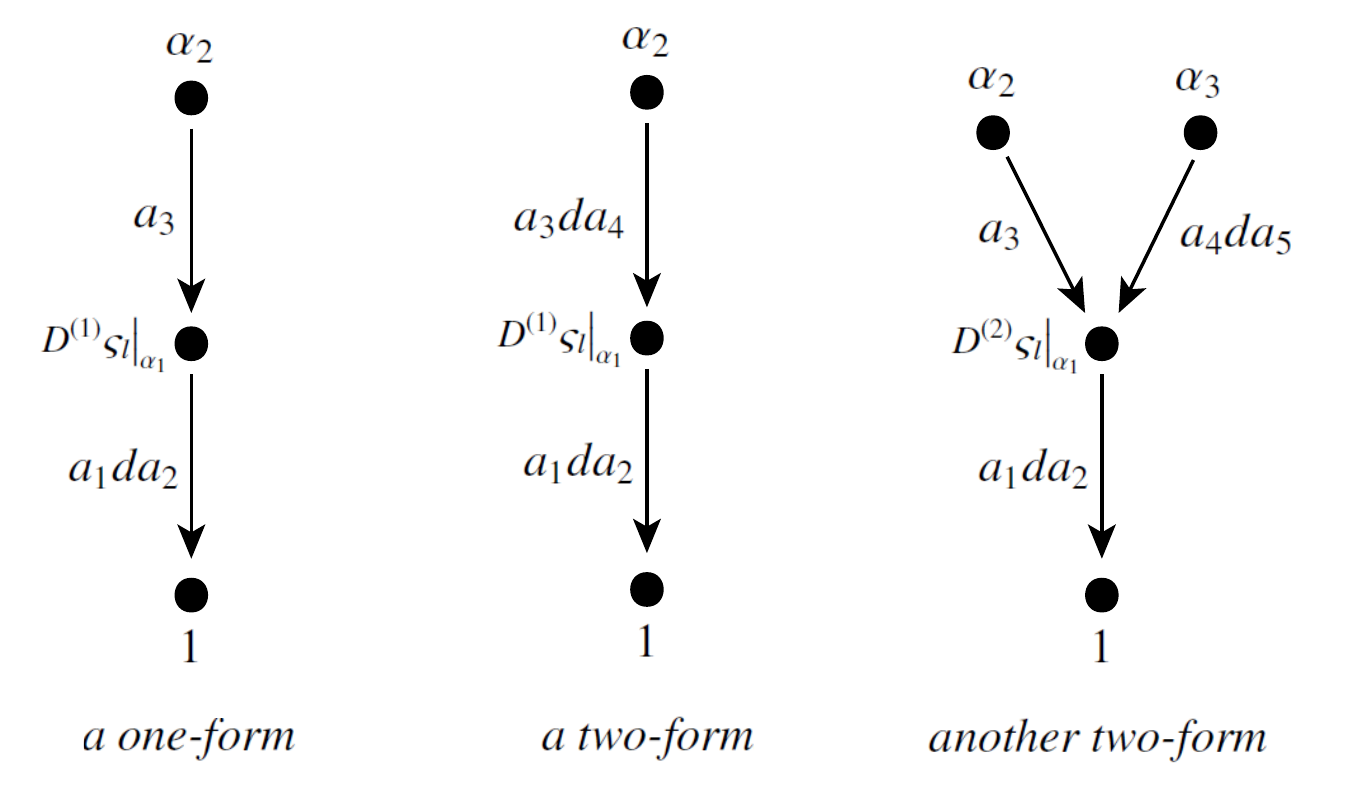}
	\caption{}
	\label{fig:form-tree}
\end{figure}

\begin{defn} \label{def:diff}
	The differential of a form over $\widetilde{\cA}$ is defined as follows.  
	
	A zero-form in the $0$-th graded piece $\alpha \in \widetilde{\cA}_0$ is simply an element in $\mathrm{Mat}_{F}(\hat{\cA})$, and its differential is given by the entriwise differential in $DR^\bullet(\hat{\cA}/B)$.
	A zero-form in the $p$-th graded piece $\alpha \in \widetilde{\cA}_p$ can be written as 
	$$ \alpha = a_0 + \sum_{k=1}^m a_k \circ \varsigma_{l(k)} \circ \alpha_k \in DR^0(\widetilde{\cA}) $$  
	where $a_k \in \mathrm{Mat}_{F}(\hat{\cA})$ for $k=0,\ldots,m$, $\alpha_k \in \widetilde{\cA}_{p-1}$, and $l(k)=1,\ldots,N$.
	Then
	$$d\alpha := da_0 + \sum_k da_k  \cdot (\varsigma_{l(k)}\circ \alpha_k)+ \sum_k a_k \cdot \left.D^{(1)}\varsigma_{l(k)}\right|_{\alpha_k} (d \alpha_k) \in DR^{1}(\widetilde{\cA})$$
	where $d \alpha_k$ has already been defined by the inductive assumption since $\alpha_k \in \widetilde{\cA}_{p-1}$.
	
	For $p$-forms with $p>0$,
	it suffices to define differential of a $p$-form attached to a form-valued tree.  For a leaf, the output is simply its label $\alpha \in \widetilde{\cA}$, whose differential has been defined above.  For a node which is neither a leaf nor the root, its output is of the form $\left. D^{(p)} \varsigma_l \right|_{\alpha} (\phi_1\cdot \eta_1,\ldots,\phi_p\cdot \eta_p)$, where $\phi_k\in DR^\bullet(\mathrm{Mat}_{F}(\hat{\cA}))$ are attached to the incoming edges, and $\eta_k$ are the outputs of the nodes adjacent to the incoming edges.  Its differential is defined as
	\begin{align*}
	&d\left(\left. D^{(p)} \varsigma_l \right|_{\alpha} (\phi_1\cdot \eta_1,\ldots,\phi_p\cdot \eta_p)\right) \\
	:=& \left.D^{(p+1)}\varsigma_l\right|_{\alpha}(d\alpha, \phi_1\cdot \eta_1,\ldots,\phi_p\cdot \eta_p) \\
	&+ \sum_{k=1}^p (-1)^{\deg (\phi_1\eta_1)+\ldots+\deg (\phi_{k-1}\eta_{k-1})}  \left. D^{(p)} \varsigma_l \right|_{\alpha}(\phi_1\cdot \eta_1,\ldots, (d\phi_k)\cdot \eta_k \\
	&+ (-1)^{\deg \phi_k} \phi_k \cdot d\eta_k ,\ldots,\phi_p\cdot \eta_p)
	\end{align*}
	where the differential $d\eta_k$ is already known by induction assumption on the generation of the tree.  The $p$-form attached to the tree is the output of the root, which is of the form $\sum_k \phi_k \cdot \eta_k$.  Its differential is defined as $\sum_k \left(d\phi_k \cdot \eta_k + (-1)^{\deg \phi_k} \phi_k \cdot d\eta_k \right)$, where $d\eta_k$ has been defined by inductive assumption.
\end{defn}

The differential of a zero-form has a nice expression in terms of a sum over sub-trees of the activation tree as follows.

\begin{prop}
	Consider $\alpha \in \widetilde{\cA}$ represented by an activation tree $T$.  Then $d\alpha \in DR^1(\widetilde{\cA})$ is a sum over all the nodes of $T$, and the terms are given as follows.  For each node, there is a unique path $\gamma_1\ldots \gamma_r$ in $T$ connecting from that node to the root, where $\gamma_k$ denotes the (oriented) edges.  (When the node is the root, the path is trivial and the corresponding term is simply $0$.)  The corresponding term equals to
	\begin{equation}
	a_{\gamma_1} \left. D^{(1)}\varsigma_{l(t(\gamma_1))}\right|_{\alpha_{t(\gamma_1)}} \ldots 
	a_{\gamma_{r-1}} \left. D^{(1)}\varsigma_{l(t(\gamma_{r-1}))}\right|_{\alpha_{t(\gamma_{r-1})}}\, da_{\gamma_r}\cdot  (\varsigma_{l(t(\gamma_r))}\circ \alpha_{t(\gamma_r)})
	\label{eq:dalpha}
	\end{equation}
	where $\alpha_i$ for a node $i$ of $T$ denotes the output at the node $i$.
\end{prop}

\begin{proof}
	The statement easily holds for the zeroth generation: the tree only has the root and leaves as nodes, and the zeroth form has an expression $\sum_i a_i$ for $a_i \in DR^\bullet(\mathrm{Mat}_{F}(\hat{\cA}))$, whose differential is simply $\sum_i d a_i$, which is a sum over the leaves.
	
	Suppose the statement holds for all elements in the $p$-th generation.  For $\alpha = a_0 + \sum_{k=1}^m a_k \circ \varsigma_{l(k)} \circ \alpha_k$ in the $(p+1)$-th generation, $d\alpha = da_0 + \sum_k da_k  \cdot (\varsigma_{l(k)}\circ \alpha_k)+ \sum_k a_k \cdot \left.D^{(1)}\varsigma_{l(k)}\right|_{\alpha_k} (d \alpha_k) \in DR^{1}(\widetilde{\cA})$, where $d\alpha_k$ is a sum over the nodes of the activation tree of $\alpha_k$ as given in Equation \eqref{eq:dalpha}.  The first term $da_0$ and second term $da_k  \cdot (\varsigma_{l(k)}\circ \alpha_k)$  correspond to the tail nodes of the edges of $a_k$ for $k=0,\ldots,m$.  Thus $d \alpha$ is a sum over all the nodes with the summands given by \eqref{eq:dalpha}.
\end{proof}

\begin{example}
	Consider the $0$-form
	$$ \alpha = a_0 + a_1 \varsigma_1 \circ (a_{1,0} + a_{1,1}\varsigma_{1,1}\circ a_{1,1,0}).$$
	Its differential equals to
	\begin{align*}
	d \alpha =& da_0 + da_1 \cdot \alpha_1
	&+a_1 \left. D^{(1)}\varsigma_1\right|_{a_{1,0} + a_{1,1}\varsigma_{1,1}\circ a_{1,1,0}} (da_{1,0} + da_{1,1} \cdot \alpha_{1,1} + a_{1,1} \left.D^{(1)}\varsigma_{1,1}\right|_{a_{1,1,0}} da_{1,1,0})
	\end{align*}
	where $\alpha_1 = \varsigma_1 \circ (a_{1,0} + a_{1,1}\varsigma_{1,1}\circ a_{1,1,0})$ and $\alpha_{1,1} = \varsigma_{1,1}\circ a_{1,1,0}$.
	It equals to the sum over the nodes of the activation tree of $\alpha$ as shown in Figure \ref{fig:differential}.
\end{example}

\begin{figure}[h]
	\centering
	\includegraphics[scale=0.35]{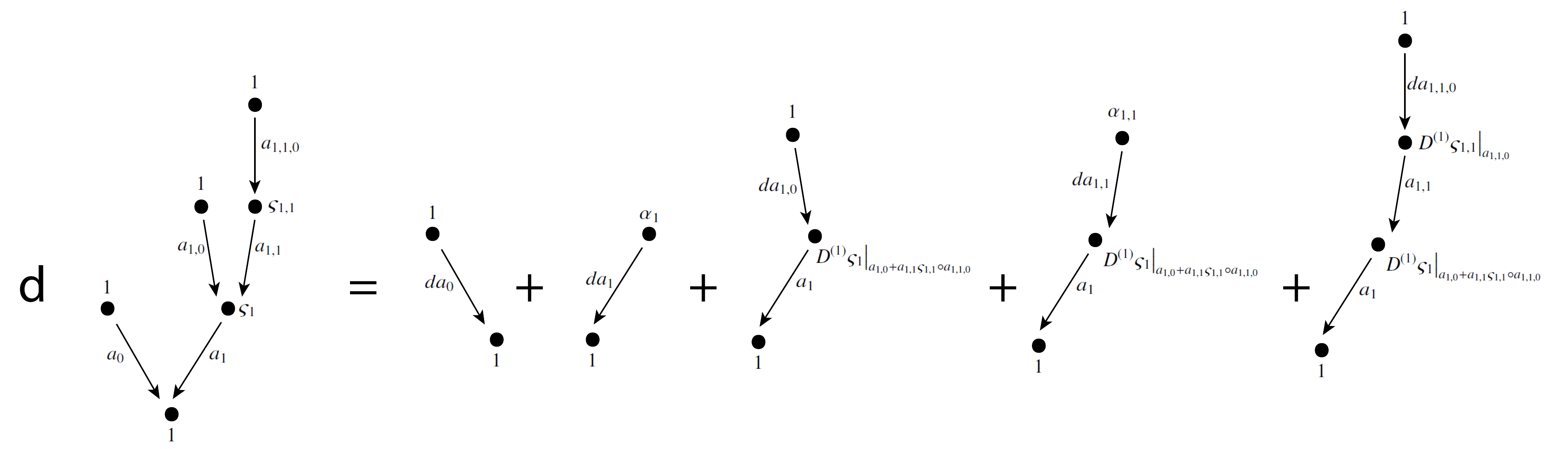}
	\caption{}
	\label{fig:differential}
\end{figure}

\begin{remark}
	The output at a node of an activation tree representing $\alpha \in \widetilde{\cA}$ can be computed by the algorithm called \emph{forward propagation}.  Namely, the previous results $\alpha_k$ (pre-activation values) have been stored in memory, and the current output is computed as $\sum_k a_k \cdot \varsigma_{l(k)} \circ \alpha_k$ (where $\varsigma_{l(k)}$ are the activation functions at previous nodes and $a_k$ are labeling the incoming edges) and stored to memory for later steps.
	
	For the differential $d\alpha$, the computation \eqref{eq:dalpha} uses the stored outputs $\alpha_i$ in the forward propagation.  Moreover, the expression $a_{\gamma_1} \left. D^{(1)}\varsigma_{l(t(\gamma_1))}\right|_{\alpha_{t(\gamma_1)}} \ldots 
	a_{\gamma_{r-1}} \left. D^{(1)}\varsigma_{l(t(\gamma_{r-1}))}\right|_{\alpha_{t(\gamma_{r-1})}}$ appears in every term of $d\alpha$ corresponding to a path in $T$ that contains $\gamma_{r-1}\ldots \gamma_1$.  Thus it is good to start with the root to compute and store the values of $a_{\gamma_1} \left. D^{(1)}\varsigma_{l(t(\gamma_1))}\right|_{\alpha_{t(\gamma_1)}} \ldots$ 
	$a_{\gamma_{r-1}} \left. D^{(1)}\varsigma_{l(t(\gamma_{r-1}))}\right|_{\alpha_{t(\gamma_{r-1})}}$, and move backward with respect to the orientation of the tree $T$.  This is well known as the \emph{backward propagation algorithm}.
\end{remark}

\begin{prop}
	$d^2 = 0$.
\end{prop}

\begin{proof}
	First consider a zero-form, that is, $\alpha \in \widetilde{\cA}$.  $\alpha$ is represented by an activation tree.  Recall that $d^2 a = 0$ for $a \in A = \mathrm{Mat}_{F}(\hat{\cA})$ is already known for differential of forms on an algebra $A$.
	
	We can write 
	$$ \alpha = a_0 + \sum_{k=1}^m a_k \circ \varsigma_{l(k)} \circ \alpha_k \in DR^0(\widetilde{\cA}) $$  
	where $a_k \in \mathrm{Mat}_{F}(\hat{\cA})$ for $k=0,\ldots,m$, $\alpha_k \in \widetilde{\cA}$ has one less generation than $\alpha$, and $l(k)=1,\ldots,N$.
	Then
	\begin{align*}
	d^2\alpha =& d\left(da_0 + da_k  \cdot (\varsigma_{l(k)}\circ \alpha_k)+ a_k \cdot \left.D^{(1)}\varsigma_{l(k)}\right|_{\alpha_k} (d \alpha_k)\right)\\ 
	=& -da_k  \cdot d(\varsigma_{l(k)}\circ \alpha_k) + da_k \cdot \left.D^{(1)}\varsigma_{l(k)}\right|_{\alpha_k} (d \alpha_k) + a_k \cdot \left.D^{(2)}\varsigma_{l(k)}\right|_{\alpha_k} (d \alpha_k, d \alpha_k).
	\end{align*}
	The first two terms cancel since $d(\varsigma_{l(k)}\circ \alpha_k)=\left.D^{(1)}\varsigma_{l(k)}\right|_{\alpha_k} (d \alpha_k)$.  The third term vanishes since $D^{(2)}\varsigma_{l(k)}$ is supersymmetric about its input (Equation \eqref{eq:Dss}).
	
	For a general $p$-form,
	it suffices to prove $d\psi=0$ for $\psi$ represented by a form-valued tree.  We will do induction on the generation of the tree.  We already know the statement when the tree is trivial (which is the case of a zero-form).   The $p$-form $\psi$ is given as
	$\psi = \sum_k \phi_k \cdot \eta_k$ for some $\phi_k \in DR^\bullet(\mathrm{Mat}_{F}(\hat{\cA}))$ and $\eta_k$ has a smaller generation than $\psi$.  Then
	$$ d^2\psi = \sum_k \left((-1)^{\deg \phi_k}d\phi_k \cdot d \eta_k +  (-1)^{\deg \phi_k + 1}d\phi_k \cdot d \eta_k + (-1)^{2 \deg \phi_k}\phi_k \cdot d^2 \eta_k \right). $$
	The first two terms cancel.  The last term vanishes by inductive assumption. 
\end{proof}

Finally, we show that differential forms on the near-ring $\widetilde{\cA}$ induce $G$-invariant $\mathrm{Map}\left(F ,F\right)$-valued differential forms over the space of framed $\cA$-modules $R$.  

\begin{theorem} \label{thm:Atildeform}
	There exists a degree-preserving map $$DR^\bullet(\widetilde{\cA}) \to  (\Omega^\bullet(R,\mathbf{Map}\left(F ,F\right)))^{G}$$
	which commutes with $d$ on the two sides, and equals to the map \eqref{eq:lin-ind-form}: $DR^\bullet(\mathrm{Mat}_{F}(\hat{\cA})) \to (\Omega^\bullet(R,\End\left(F\right)))^{G}$ when restricted to $DR^\bullet(\mathrm{Mat}_{F}(\hat{\cA}))$.
	Here, $\mathbf{Map}\left(F ,F\right)$ denotes the trivial bundle $\mathrm{Map}\left(F ,F\right) \times R$, and the action of $G=\GL(V)$ on fiber direction is trivial.
\end{theorem}

\begin{proof}
	First consider the case of a zero-form.  We associate $\alpha \in DR^0(\widetilde{\cA})$ to a $G$-invariant $\mathrm{Map}(F,F)$-valued function over $R$ inductively on its generation as an element in $\widetilde{\cA}$.  
	In the zeroth generation, it is just an element in $\mathrm{Mat}_{F}(\hat{\cA})$, which induces a matrix whose entries lie in $\Omega^0(R)^G$ by Proposition \ref{prop:lin-ind-form}.  This gives a self-map $F \to F$ over $[R/G]$.
	If $\alpha$ is in the $p$-th generation, then it is written as $\alpha = a_0 + \sum_{k=1}^m a_k \circ \varsigma_{l(k)} \circ \alpha_k \in DR^0(\widetilde{\cA})$, where $\alpha_k$ is in the $(p-1)$-th generation and induces a self-map $F \to F$ over $[R/G]$.  By composing with the corresponding functions $\sigma_{l(k)}:F\to F$ and the induced functions of $a_k \in \mathrm{Mat}_{F}(\hat{\cA})$, we obtain a self-map $F \to F$ over $[R/G]$ corresponding to $\alpha$.
	
	For a $k$-form $\psi \in DR^\bullet(\widetilde{\cA})$, we do an induction on the generation of its corresponding form-valued tree to associate it with a $G$-invariant $\mathrm{Map}\left(F ,F\right)$-valued $k$-form over $R$.  In th zeroth generation it must be a zero-form (where the associated form-valued tree is simply a single node), which is done by the previous paragraph.  In general $\psi = \sum_k \phi_k \cdot \eta_k$ for some $\phi_k \in DR^\bullet(\mathrm{Mat}_{F}(\hat{\cA}))$ and $\eta_k$ has a smaller generation than $\psi$. Both $\phi_k$ and $\eta_k$ have been associated with $G$-invariant $\mathrm{Map}\left(F ,F\right)$-valued $k$-forms.  Then their matrix products (and by wedge product entriwise) give the required $k$-form associated to $\psi$.
	
	It follows from the chain rule that the differential for $DR^\bullet (\widetilde{\cA})$ given in Definition \ref{def:diff} agrees with that for $\mathrm{Map}\left(F ,F\right)$-valued forms over $R$.  Moreover, for $\phi \in DR^\bullet(\mathrm{Mat}_{F}(\hat{\cA}))$, it is in the first generation written as $\phi \cdot 1$.  By the above definition, the association is given by the map \eqref{eq:lin-ind-form}.
\end{proof}
	
So far, this gives matrix-valued differential forms on $[R/G]$.  To produce $\C$-valued forms, that is, to remove the component $\mathrm{Map}\left(F ,F\right)$ in the above theorem, we proceed as follows.
The near-ring $\widetilde{\cA}$ can be augmented with the inclusion and projection symbols $\iota_i$ and $p_j$, where $\iota_i$ represents the inclusion $\C \to F$ of the $i$-th coordinate axis, and $p_j$ represents the projection $F \to \C$ in the $i$-th direction.  This forms an augmented near-ring
$$ \bigoplus_{k=1}^\infty \left(\{p_1,\ldots, p_n\}\circ \widetilde{\cA} \circ \left(\C\cdot \{\iota_1,\ldots,\iota_n\} \right)\right)^k $$
consisting of linear combinations of elements $\left(p_i \circ \alpha \circ \left(\sum_{j} x_j \iota_j \right)\right)^k$ for $\alpha \in \widetilde{\cA}$,
with the relations $p_i \circ \iota_j = \delta_{ij} \cdot 1$ and $\iota_j \circ 1_{\widetilde{\cA}} \circ p_i = \delta_{ij} \cdot 1$.
Then differential forms in this augmented near-ring induces $G$-invariant differential forms in $ (\Omega^\bullet(R))^{G}$.  The proof is similar and we shall not repeat.   

In application, we fix an algorithm $\tilde{\gamma} \in \widetilde{\cA}$ and consider 
$$\varphi^{\tilde{\gamma}}(x) = \left(p_i \circ \tilde{\gamma} \circ \left(\sum_{j} x_j \iota_j \right) \right)_{i=1}^n$$ 
for each element $x=(x_1,\ldots,x_n) \in F$.  $\varphi^{\tilde{\gamma}}(x)$ is a vector whose entries are elements inside the above augmented near-ring. Given $f: K \to F$, we have 
\begin{equation*}
\int _{K}\left| \varphi^{\tilde{\gamma}}(x)-f\left(x\right)\right|^{2} dx
\end{equation*}
which is a 0-form on the augmented near-ring.  This 0-form and its differential induces the cost function and its differential on $[R/G]$ respectively, which are the central objects in machine learning.

\section{Uniformization}\label{section:unif}

In this section, we apply the idea of uniformization of metrics on framed quiver moduli spaces, which are interpreted as moduli of computing machines from the previous section.

The uniformization theorem for Riemann surfaces was a big discovery of Klein, Poincar\'e and Koebe in the 19th century.  It asserts that every simply connected Riemann surface is conformally equivalent to either the complex plane, the Riemann sphere, or the hyperbolic disc.

Such a classification also holds for Riemannian symmetric spaces.  Namely, any irreducible simply connected symmetric space is either of Euclidean type, compact type, and non-compact type, depending on whether its sectional curvature is identically zero, non-negative, or non-positive.

As a key example, $\Gr(n,d)$ is a compact Hermitian symmetric space. It has a non-compact dual which embeds as an open subset of $\Gr(n,d)$.  This is the celebrated Borel embedding, and was uniformly studied for symmetric R-spaces and generalized Grassmannians in \cite{CHL-Borel}.   The non-compact dual to $\Gr(n,d)$ is the "space-like Grassmannian" which can be thought of as a generalization of hyperbolic space.  

We generalize this to framed quiver varieties.  The key idea is that different types of quiver varieties will arise by considering space-like representations with respect to different choices of quadratic forms on the framing.  As explained in the Introduction, our motivation is to find a relation between our formulation of neural network and the original Euclidean formulation.  Using this construction, we not only get an interpolation between these two different formulations, but also find a non-compact type quiver varieties which can also be used in machine learning.  Such a family of quiver varieties of different types is what we refer as uniformization of framed quiver varieties in the title.

\subsection{A quick review}
Let $Q$ be a directed graph.  Denote by $Q_0, Q_1$ the set of vertices and arrows respectively.  A quiver representation $w$ with dimension vector $d \in \Z_{\geq 0}^{Q_0}$ associates each arrow $a$ with a matrix $w_a$ of size $d_{h(a)}\times d_{t(a)}$ (where $h(a), t(a)$ denote the head and tail vertices of $a$ respectively).  The set of complex quiver representations with dimension $\vec{d}$ form a vector space denoted by $R_{\vec{d}}(Q)$. $\GL(d):=\prod_{i\in Q_0}\GL(d_i, \C)$ acts on $R_{\vec{d}}(Q)$ via
\begin{equation}
g\cdot (w_a: a\in Q_1) = (g_{h(a)}\cdot w_a \cdot g_{t(a)}^{-1}: a \in Q_1).
\label{eq:GL(d)}
\end{equation}

Let $d, n \in \Z_{\geq 0}^{Q_0}$.  $n$ will be the dimension vector for the framing, which is a linear map $e^{(i)}: \C^{n_i} \to V_i$ at each $i\in Q_0$ (where $V_i = \C^{d_i}$).

\begin{theorem}[\cite{Nakajima-proc}]
The vector space of framed representations is given by
$$R_{n,d} = R_{d} \times \bigoplus_{i\in Q_0}\Hom(\C^{n_i}, \C^{d_i}).$$
It carries a natural action of $\GL(d)$ given by $g \cdot (w, e) = (g \cdot w, (ge^{(i)}: i\in Q_0) )$, where $g \cdot V$ is given by Equation \eqref{eq:GL(d)}. $(w, e) \in R_{n,d}$ is called \textit{stable} if there is no proper subrepresentation $U$ of $w$ which contains $\mathrm{Im} \,e$. The set of all stable points of $R_{n,d}$ is denoted by $R_{n,d}^s$.
Then the quotient $\cM_{n,d} := R_{n,d}^s / \GL(d)$ is a smooth variety, which is called to be a framed quiver moduli.
\end{theorem}



The topology of $\cM_{n,d}$ is well-understood.  Let's make an ordering of the vertices.  Namely the vertices are labeled by $\{1,\ldots,N\}$, such that $i<j$ implies there is no arrow going from $j$ to $i$.  Such a labeling exists if $Q$ has no oriented cycle.

\begin{theorem}[Reineke \cite{Reineke}] \label{thm:Reineke}
Assume $Q$ has no oriented cycle.  Consider the chain of iterated Grassmannian bundles
$M^{(N)}\stackrel{p_N}{\to} M^{(N-1)} \stackrel{p_{N-1}}{\to} \ldots \stackrel{p_{2}}{\to} M^{(1)} \stackrel{p_{1}}{\to} \pt$
(where $\pt$ denotes a singleton)
defined by induction:
$$M^{(i)}=\Gr_{M^{(i-1)}}\left(\underline{\C^{n_{i}}}\oplus\bigoplus_{j\to i}p_{i-1}^*\dots p_{j+1}^*(S_j), d_{i}\right) \to M^{(i-1)},$$
where $S_i$ denotes the tautological bundle on $M_i$ (as a Grassmannian bundle over $M_{i-1}$).  (The direct sum is over each arrow $j\to i$.)
Then $\cM_{\vec{n},\vec{d}}\cong M^{(N)}$, with universal bundles $\mathcal{V}_i\cong p_N^*\dots p_{i+1}^*S_i$ for all $i\in Q_0$.
\end{theorem}


In the previous paper \cite{JL} we introduced a Hermitian metric $H_i$ for each of these $\mathcal{V}_i$ and showed that its Ricci curvature induces a K\"ahler metric on $\mathcal{M}$. Let's quickly review this construction.

\begin{theorem}[\cite{JL}]
\label{theorem:metric}
Let $Q$ be a finite quiver. Let $R_{n,d}$ be the space of framed quiver representations of $Q$ with representing dimension $d$ and framing dimension $n$. For any path $\gamma$ in $Q$, let $e^{t(\gamma)}$ be the framing map associated to the vertex $t(\gamma)$ and let $w_{\gamma}$ be the matrix representation of $\gamma$.

For a fixed vertex $(i)$, let $\rho_i$ be the row vector whose entries are all the elements of the form $w_{\gamma}e^{t(\gamma)}:R_{n,d}\to \Hom(\C^{n_{t(\gamma)}}, \C^{d_i})$ such that $h(\gamma)=i$. Consider 

\begin{equation}\label{equation:metric}
\rho_i\rho_i^*=\sum\limits_{h(\gamma)=i}\left(w_{\gamma}e^{t(\gamma)}\right)\left(w_{\gamma}e^{t(\gamma)}\right)^*
\end{equation} 
as a map $\rho_i\rho_i^*:R_{n,d}\to \text{End}(\C^{d_i})$.

Then $(\rho_i\rho_i^*)^{-1}$ is $\GL(d)$-equivariant and descends to a Hermitian metric on $\mathcal{V}_i$ over $\mathcal{M}$. We denote this resulting metric as $H_i$.

Suppose $Q$ has no oriented cycle. Then
\begin{equation}
H_T:=\sum_i\partial\bar{\partial}\log\det H_i=\sum_i\left(tr(\partial\rho_i)^*H_i\partial\rho_i-tr\left(H_i\rho_i(\partial\rho_i)^*H_i(\partial\rho_i)\rho_i^*\right)\right)
\end{equation}
defines a K\"ahler metric on $\mathcal{M}$.
\end{theorem}

\subsection{Illustration by examples}

First, consider the simplest possible example, namely the quiver with a single vertex.

\begin{example} \label{ex:A1}
	Let $Q$ consist of a single vertex $(1)$ with no arrows. Let the representing dimension and the framing dimension be $d$ and $n$ respectively where $d<n$. The framed quiver moduli is simply $\Gr(n,d)$, the Grassmannian of surjective linear maps $\C^n \to \C^d$. Equation \ref{equation:metric} becomes $H=ee^*$ on the universal bundle over the dual Grassmannian. If we take the chart where the first $d$-many components of $e$ form an invertible map, we can rewrite $e$ as $e=(Id_d, b)$ due to the $G_d$-equivalence. Then $H$ becomes $(Id_d+b)^{-1}$, the standard metric on the universal bundle over $\Gr(n,d)$. In particular, for $d=1$, $\Gr(n,1)$ is the projective space $\mathbb{P}^{n-1}$, and the Ricci curvature of $H$ is the Fubini-Study metric.
\end{example}

\begin{figure}[h]
	\begin{tikzcd}
		& (3) \arrow[dr, "a_4"] & \\ (1) \arrow[dr, "a_1"]\arrow[ur, "a_2"] & & (4)\\ & (2) \arrow[ur, "a_3"] &
	\end{tikzcd}
	\caption{}
	\label{fig:A4}
\end{figure}

Now consider a slightly more involved example.  Let $Q$ be the quiver depicted in Figure \ref{fig:A4}. Thus the vertex set is $\{(1), (2), (3), (4)\}$ and the arrow set is $\{a_1:(1)\to (2), a_2:(1)\to (3), a_3:(2)\to (4), a_4:(3)\to (4)\}$. We define $\mathcal{A}$ to be the path algebra of $Q$ over $\C$. That is, an element of $\mathcal{A}$ is a formal sum over $\C$ generated by the paths on the underlying directed graph of $Q$. In addition, we associate to $Q$ a vector space $V=\C^{d}$ and a framing space $F=\C^{n}$ with some kind of decomposition as $F=F_{\mathrm{in}}\oplus F_2\oplus F_3\oplus F_{\mathrm{out}}$. In this setting, we take $(1)$ as our input vertex, $(4)$ as our output vertex, and vertices $(2)$ and $(3)$ as "middle" or "memory" vertices, so we think of $F_{in}$ as being associated to $(1)$, $F_{\mathrm{out}}$ to $(4)$, and $F_m\cong F_2\oplus F_3$ with $(2)$ and $(3)$ respectively. In addition, we want a decomposition of $V$ as $V=V_1\oplus V_2\oplus V_3\oplus V_4$ with $V_i$ associated to vertex $(i)$ for each $i$. We endow this space with a $GL(V)$-equivariant family of Hermitian metrics $\bigoplus_iH_i$: (Theorem \ref{theorem:metric}):

$$\rho_i\rho_i^*=\sum\limits_{h(\gamma)=i}\left(w_{\gamma}e^{t(\gamma)}\right)\left(w_{\gamma}e^{t(\gamma)}\right)^*.$$

Take activation functions $\sigma_j^F:F_m\to F_m$. We have an activation module (Definition \ref{def:cm}). The framing maps are of the form $(e_{\mathrm{in}}, e_2, e_3, e_{\mathrm{out}})$ where $(e_{\mathrm{in}})_j:F_{\mathrm{in}}\to V_1$, $(e_2)_j:F_2\to V_2$, $(e_3)_j:F_3\to V_3$, and $(e_{\mathrm{out}})_j:F_{\mathrm{out}}\to V_4$.  We can be extend by zero and encode all of them as $e:F\to V$. Furthermore, we can require the maps $\sigma_j^F$ to decompose as $(\sigma_j^F)_2\oplus(\sigma_j^F)_3$ where $(\sigma_j^F)_i$ is a map from $F_i\to F_i$.

Now, we want to choose an algorithm, that is, an element $\tilde{\gamma}\in\tilde{\mathcal{A}}$ that starts at $F_{\mathrm{in}}$ and ends at $F_{\mathrm{out}}$. This takes the form \[\tilde{\gamma}=e_{\mathrm{out}}^*(a_4e_m\sigma_2e_m^*a_2e_{\mathrm{in}}+a_3e_m\sigma_3e_m^*a_1)e_{\mathrm{in}}.\]

\begin{figure}[h]
	\centering
	\includegraphics[width=\linewidth]{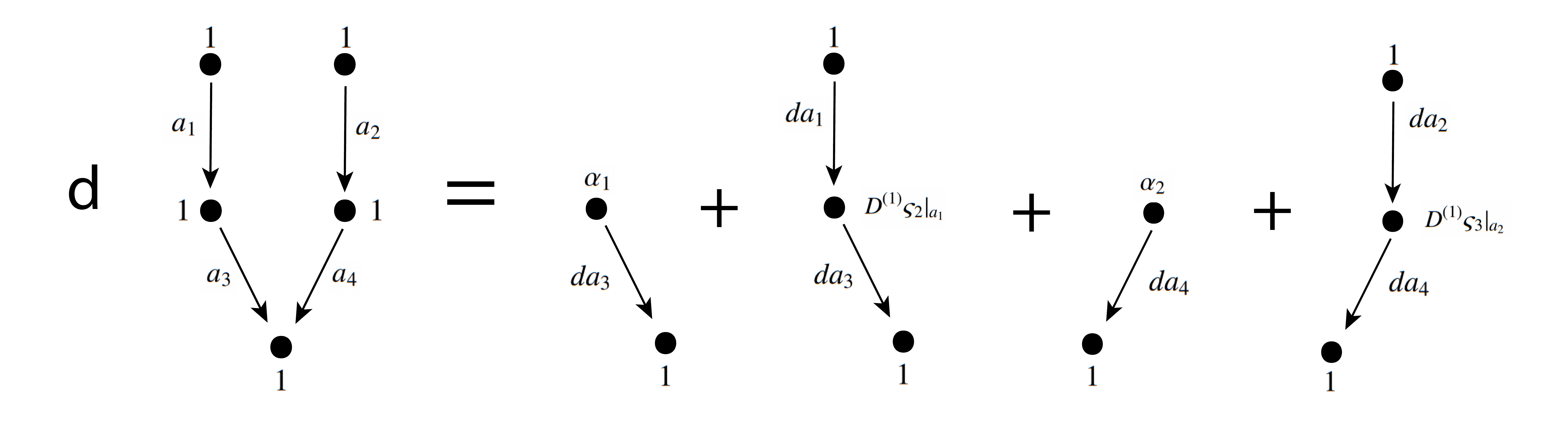}
	\caption{}
	\label{fig:ex-tree}
\end{figure}

The activation tree for this algorithm is shown in Figure \ref{fig:ex-tree} on the left. The rest of Figure \ref{fig:ex-tree} is the form-valued tree derived from the activation tree as in Figure \ref{fig:differential}. This activation tree gives the forward propagation used in deep learning for actual computation. Just as vital, the back propagation attained from the form-valued tree is used in order to train the network.

\subsection{The non-compact dual of framed quiver moduli}\label{subsection:hyper}

Assume that $n_i\geq d_i$ $\forall i$. We write the framing map as $e^{(i)}=(\epsilon_i$ $b_i)$ where $\epsilon_i$ and $b_i$ are respectively the "basis part" and "bias part of our framing map $e^{(i)}$. Then Equation \ref{equation:metric} can be modified to be:
\begin{equation}\label{eq:uniform}
H_i^{\alpha}=\left(\epsilon_i\epsilon_i^*+\alpha b_ib_i^*+\sum_{\gamma:h(\gamma)=i, \gamma\neq\emptyset}\alpha_{\gamma}w_{\gamma}e^{t(\gamma)}\left(w_{\gamma}e^{t(\gamma)}\right)^*\right)^{-1}.
\end{equation}

It is this generalization of the metric which we use for the uniformization. By varying $\alpha$ and $\alpha_{\gamma}$, we get different quadratic forms. For example, in Equation \ref{equation:metric}, $\alpha$ and all $\alpha_{\gamma}$ are simply $1$. The zero curvature case will elaborated on later in Section \ref{section:Euclidean}.

\begin{remark}
	The application of hyperbolic geometry has mostly focused on fiber direction in existing literature, namely the representation spaces (and their corresponding universal bundles over the moduli).  Here, we are concerned about metrics on the moduli space (playing the role of the weight space).  It is general for all quiver moduli, not just restricted to specific models.  Thus, in this moduli approach, the method of varying metrics (with positive, zero or negative curvatures) can be applied to any model of machine learning.
\end{remark}

For now we will set the $\alpha$ and $\alpha_{\gamma}$ to -1 to consider the negative curvature case. Namely,
\begin{equation}\label{equation:hyperbolic}
H_i^-:=\left(\epsilon_i\epsilon_i^*-b_ib_i^*-\sum_{\gamma:h(\gamma)=i, \gamma\neq\emptyset}w_{\gamma}e^{t(\gamma)}\left(w_{\gamma}e^{t(\gamma)}\right)^*\right)^{-1}.
\end{equation}
It must be emphasized that this quadratic form is \textbf{not} positive-definite on $\mathcal{V}_i$ and thus cannot serve as a metric.

The brilliant idea here is that we restrict to the subset of the moduli space where this quadratic form is positive-definite and thus gives a metric. This restriction gives the non-compact dual of the framed quiver moduli. As before, let's consider the $A_1$-quiver and what this metric looks like on that quiver in particular.

\begin{example}\label{example:space}
Let $Q$ consist of a single vertex with no arrows. Let the framing dimension of the representation space be $n$, and suppose the representing dimension at the single vertex be 1. Equation \ref{equation:hyperbolic} becomes
\[H^-=(|\epsilon|^2-|b|^2)^{-1}.\]

Since we restrict to the subset where $H^-$ is positive-definite, $\epsilon$ needs to be nonzero. By applying the quiver automorphism, $\epsilon$ can be rescaled to be 1.  Thus $H^-=(1-|b|^2)^{-1}$, and $|b|^2<1$. This gives the hyperbolic moduli, which is the open unit ball in $\C^{n-1}$. The Ricci curvature of $H^-$ gives the Poincar\'e metric.
\end{example}

Thus from Examples \ref{ex:A1} and \ref{example:space} we can see the motivating duality mentioned at the start of the section.

\begin{defn}\label{def:hyper}
Assume $Q$ has no oriented cycle. Let $\rho_i$ be as in Definition \ref{theorem:metric} so that $\rho_i$ is a row vector with entries of the form $w_{\gamma}e^{t(\gamma)}$ where $\gamma$ is some path in $Q$ ending at vertex $(i)$, $w_{\gamma}$ is the representing matrix of this path, and $e^{t(\gamma)}$ is the framing map at $t(\gamma)$, the starting vertex of $\gamma$. Arrange the entries of $\rho_i$ so that the first $n_i$-many entries correspond to the framing arrows at vertex $(i)$. Then let $H_i^-$ be the quadratic form defined by:

\begin{equation} \label{eq:H_i^-}
H_i^-=\left(\rho_i\left(\begin{matrix} I_{d_i} & 0 \\0 & -I_{N_i-d_i}\end{matrix}\right)\rho_i^*\right)^{-1}
\end{equation}

Here, $N_i=\sum\limits_{(j)\rightsquigarrow (i)}n_j$. We define $R_{n,d}^-$ to be the subset of $R_{n,d}$ where $H_i^-$ is positive-definite for all $i$.
\end{defn}

In above, the notation $(j)\rightsquigarrow (i)$ means all paths starting from $(j)$ and ending at $(i)$, not just single arrows. In particular, $n_j$ gets counted once for each distinct path from $(j)$ to $(i)$. Note that $N_i\geq d_i$ $\forall i$ when $\mathcal{M}\neq\emptyset$, which we always assume to be the case.

\begin{prop} \label{prop:R-}
	\[  R_{n,d}^- \subset R_{n,d}^s.\]
\end{prop}

\begin{proof}
	Consider a point in $R_{n,d}^-$.  Write $\rho_i = (\epsilon_i \,\, R)$ evaluated at this point as a $(d_i \times N_i)$-matrix, where $\epsilon_i$ is a $(d_i \times d_i)$-matrix and $R$ is the remaining part.  Then $H_i^- = (\epsilon_i \epsilon_i^* - RR^*)^{-1}$.  We claim that $\epsilon_i$ must be invertible, and hence $\rho_i$ is surjective.  This is true for all $i$, and hence the point is stable.
	
	Suppose $\epsilon_i$ is not invertible.  Then there exists $v$ such that $\epsilon_i^*\cdot v=0$.  Then $v^* H_i^- v = - v^* RR^* v \leq 0$, contradicting that $H_i^-$ evaluated at each point in $R_{n,d}^-$ is positive-definite.
\end{proof}

\begin{lemma}\label{lemma:hypequiv}

$H_i^-$ is $G_d$-equivariant and $R_{n,d}^-$ is $G_d$-invariant.

\end{lemma}

\begin{proof}

$H_i^-$ is $GL(d)$-equivariant because

\[\left((g\cdot w_{\gamma}e^{t(\gamma)})(g\cdot w_{\gamma}e^{t(\gamma)})^*\right)^{-1}=g^{-1}\left((w_{\gamma}e^{t(\gamma)})(w_{\gamma}e^{t(\gamma)})^*\right)^{-1}(g^*)^{-1}.\]

The reason that $R_{n,d}^-$ is $GL(d)$-invariant is because if $x^*\left((w_{\gamma}e^{t(\gamma)})(w_{\gamma}e^{t(\gamma)})^*\right)^{-1}x>0$, then

 \[(g\cdot x)^*\left((g\cdot w_{\gamma}e^{t(\gamma)})(g\cdot w_{\gamma}e^{t(\gamma)})^*\right)^{-1}(g\cdot x)=x^*\left((w_{\gamma}e^{t(\gamma)})(w_{\gamma}e^{t(\gamma)})^*\right)^{-1}x>0\]
 
 by the $GL(d)$-equivariance of $H_i^-$. Thus, action by $GL(d)$ sends $R_{n,d}^-$ to itself.

\end{proof}

\begin{defn}
If $n_i\geq d_i$, then $e^{(i)}$ can be written as $e^{(i)}=(\epsilon_i, b_i)$ where $\epsilon_i$ is the $d_i$-many components of $e^{(i)}$ and $b_i$ is the remaining $(n_i-d_i)$-many components. We call $\epsilon_i$ to be the basis part of $e^{(i)}$ and $b_i$ to be the bias part of $e^{(i)}$.
\end{defn}

We call $\epsilon_i$ the basis part because we think of it as imposing a basis on $V_i$.

From now on, we will assume $n_i \geq d_i$, which is the case in applications.  This assumption also ensures that the choices of negative signs in defining $H_i^-$ for different vertices $i$ are compatible, so that $R_{n,d}^- \not= \emptyset$.


\begin{prop}\label{prop:invertible}
Assume that $n_i\geq d_i$ for all $i$. \[\emptyset \not=R_{n,d}^-\subset\{\epsilon_i \text{ is invertible for all }i\} \subset R_{n,d}^s.\]

\end{prop}

\begin{proof}
From the proof of Proposition \ref{prop:R-}, it is clear that $\epsilon_i$ is invertible over $R_{n,d}^-$ and these points belong to $R_{n,d}^s$.  To see that $R_{n,d}^- \not= \emptyset$, we can take $\epsilon_i = \Id$ and $b_i=0$ for all $i\in Q_0$, and all the representing matrices for the arrows of $Q$ to be $0$.  This gives a point in $R_{n,d}$ at which $H_i^- = \Id$ is positive-definite.
\end{proof}

Suppose a Lie group $G$ acts on a vector bundle $V\overset{\pi}{\to}M$ equivariantly fiberwise linearly, and the action of $G$ on $M$ is free and proper. A metric $H$ on $V$ is $G$-equivariant if \[H_x(v, w)=H_{g\cdot x}(c\cdot v, g\cdot w).\] It is possible that $V$ may not descend to a vector bundle over $M/G$ if $G_p\subset G$ acts on $V$ non-trivially at a point $p\in M$. In the case that the corresponding bundle does exist, $H$ will descend to that bundle if and only if $H$ is $G$-equivariant.

Since we know that $R_{n,d}^-$ is a $\GL(d)$-invariant non-compact open subset, we can quotient by $\GL(d)$ in the same way we do for $R_{n,d}^s$.

\begin{defn}
We define the dual of $\mathcal{M}$ as the quotient $\mathcal{M}^-=R_{n,d}^-/\GL(d)$ with universal bundles $\mathcal{V}_i^-:=(R_{n,d}^-\times \C^{d_i})/\GL(d_i)$. Since $H_i^-$ is Hermitian and $G_d$-equivariant, it descends to a metric on $\mathcal{V}_i^-$ over $\mathcal{M}^-$.
\end{defn}

\begin{remark}\label{remark:basis}
As a result of Proposition \ref{prop:invertible} and the fact that $\GL(d)$ acts only on the left on the framing space, $e^{(i)}=(\epsilon_i, b_i)=(I_{d_i}, \tilde{b}_i)$ where $\tilde{b}_i=\epsilon_i^{-1}b_i$ and is itself a generic bias vector for each $i$. Thus, from this point forward we will be assuming both that $n_i\geq d_i$ for all $i$ and that all framing maps are of the form $e^{(i)}=(\Id, b_i)$. Thus $\mathcal{M}^-\subset R_{n-d,d}$.
\end{remark}



\begin{example} \label{example:A1}
Consider the framed $A_1$ quiver (the quiver with one vertex and zero arrows). Let $d\leq n=N$. Then $\mathcal{M}$ is $\Gr(n,d)$. As a Hermitian symmetric space, this is dual to the space-like Grassmannian $\Gr^-(n,d)$. Here, we define $\Gr^-(n, d)$ to be the open subset of $\Gr(n,d)$ consisting of $d$-planes in $\C^n$ where the quadratic form \[Q(x, y)=\sum_{i=1}^d\overline{x}_iy_i-\sum_{j=d+1}^n\overline{x}_jy_j\] is positive-definite.

Similar to Remark \ref{remark:basis}, we can take elements of $\Gr^-(n, d)$ to be of the form $(\Id, b)$ where the first $d$-many columns are the $d\times d$ identity matrix $b$ is the remaining $d\times(n-d)$ columns. Then we can say that $\Gr^-(n, d)$ is the set $\{b\in\C^{d\times(n-d)}:\Id-bb^*\geq 0\}$.

Going back to the quiver, since there is no other arrow, we see that $H_1^-=(\Id-bb^*)^{-1}$. Thus, $\mathcal{M}^-$ is going to be the set $\{b:\Id- bb^*\geq 0\}$.


In particular, when $d=1$, $\Gr(n,1)^-$ is complex hyperbolic space and the Ricci curvature of $H_1^-=\frac{1}{1-|b|^2}$ is the standard metric for the Poincare disk model of complex hyperbolic space.
\end{example}

Now we define an explicit metric on $\cM^-$, using \eqref{eq:H_i^-} written in terms of paths in $Q$, in an analogous way as the one given in Theorem \ref{theorem:metric}.

\begin{theorem}\label{corollary:Kahler}
	Assume $Q$ is acyclic.
	Define $H_T^-:=-i\sum\limits_{i}\partial\overline{\partial}\log\det H_i^-$ on $\mathcal{M}^-$. Then $H_T^-$ is a K\"ahler metric on $\mathcal{M}^-$.
\end{theorem}

\begin{proof}
	This proof is similar to that of Theorem 3.15 in \cite{JL}. We include the details for the reader's convenience.
	
	Let's denote $\rho=\rho^{(i)}=\left(w_{\gamma}e^{(t(\gamma)}\right)_{\gamma:h(\gamma)=i}$ which is a matrix-valued function on $R_{n, d}^-$. At each point of $R_{n, d}^-$, we have that $\rho$ is a linear map from $W_i:=\bigoplus\limits_{j\leftsquigarrow i}\C^{n_i}$ to $V_i$. The Ricci curvature of the metric $H_i^-$ is given by $i\partial\overline{\partial}\log\det\rho A\rho^*$ where $A$ is the matrix $\text{diag}(1, -1, -1,\dots, -1)$. Let $B$ be the matrix $\text{diag}(1, \sqrt{-1},\dots, \sqrt{-1})$ and define $\hat{\rho}:=\rho B$ so that $\hat{\rho}\hat{\rho}^*=\rho A\rho^*$. Thus we have that $H_i^-=(\hat{\rho}\hat{\rho}^*)^{-1}$.
	
	We can take the singular valued decomposition of $\hat{\rho}$ to write it as \[\hat{\rho}=U\cdot(\begin{matrix}\text{diag}(\lambda_1, \dots, \lambda_{d_i}) & 0\\\end{matrix})\cdot V^*\] where $U\in U(d_i)$, $V\in U(\dim W_i)$, and the $\lambda_i$ are all positive real numbers. We know that none of the $\lambda_i$ are zero since that would make corresponding quiver representations non-surjective and thus unstable. Then 
	\[\hat{\rho}=U\cdot(\begin{matrix}\text{diag}(\lambda_1, \lambda_2,\dots, \lambda_{d_i}) & 0\\\end{matrix})\cdot V^*,\]
	\[\hat{\rho}\hat{\rho}^*=U\left(\text{diag}(\lambda_1^2, \lambda_2^2,\dots, \lambda_{d_i}^2)\right)U^*,\]
	\[\hat{\rho}^*(\hat{\rho}\hat{\rho}^*)^{-\frac{1}{2}}=V\left(\begin{matrix} \text{diag}(\lambda_1, \dots, \lambda_{d_i})\\ 0\\\end{matrix}\right)(\text{diag}(\lambda_1^{-1},\dots, \lambda_{d_i}^{-1})U^*=V\left(\begin{matrix} I_{d_i}\\ 0\\\end{matrix}\right)U^*.\]
	
	Let us consider the decomposition $W_i=(\Im\hat{\rho}^*)\oplus(\Im\hat{\rho}^*)^{\perp}$. In particular, this shows that $\hat{\rho}^*(\hat{\rho}\hat{\rho}^*)^{-\frac{1}{2}}$ is the orthogonal embedding of $V_i$ to $\Im\hat{\rho}^*\subset W_i$.
	
	Then
	\[\partial\overline{\partial}\log\det\hat{\rho}\hat{\rho}^*=\partial\left(\tr \left((\hat{\rho}\hat{\rho}^*)^{-1}\overline{\partial}(\hat{\rho}\hat{\rho}^*)\right)\right)\]
	\[=\tr\left(\partial\left((\hat{\rho}\hat{\rho}^*)^{-1}(\hat{\rho})(\partial\hat{\rho})^*\right)\right)\]
	\[=\tr\left((\hat{\rho}\hat{\rho}^*)^{-1}(\partial\hat{\rho})(\partial\hat{\rho})^*+\left(\partial(\hat{\rho}\hat{\rho}^*)^{-1}\right)\hat{\rho}(\partial\hat{\rho})^*\right)\]
	\[=\tr\left((\partial\hat{\rho})^*(\hat{\rho}\hat{\rho}^*)^{-1}(\partial \hat{\rho})\right)-\tr\left((\hat{\rho}\hat{\rho}^*)^{-1}(\partial(\hat{\rho}\hat{\rho}^*))(\hat{\rho}\hat{\rho}^*)^{-1}\hat{\rho}(\partial\hat{\rho})^*\right)\]
	\[=\tr\left((\partial\hat{\rho})^*(\hat{\rho}\hat{\rho}^*)^{-1}(\partial \hat{\rho})\right)-\tr\left((\hat{\rho}\hat{\rho}^*)^{-1}\hat{\rho}(\partial\hat{\rho})^*(\hat{\rho}\hat{\rho}^*)^{-1}(\partial \hat{\rho})\hat{\rho}^*\right)\]
	\[=\tr\left((\partial\hat{\rho})^*(\hat{\rho}\hat{\rho}^*)^{-1}(\partial \hat{\rho})\right)-\tr\left(\left((\partial\hat{\rho})\cdot\left(\hat{\rho}^*(\hat{\rho}\hat{\rho}^*)^{-\frac{1}{2}}\right)\right)^*(\rho\rho^*)^{-1}\left((\partial\hat{\rho})\cdot\left(\hat{\rho}^*(\hat{\rho}\hat{\rho}^*)^{-\frac{1}{2}}\right)\right)\right).\]
	
	Consider a vector $v\in T^{1, 0}R^-_{n, d}\cong TR^-_{n, d}$. We can see that the term $\tr\left((\partial_v\hat{\rho})^*(\hat{\rho}\hat{\rho}^*)^{-1}(\partial_v\hat{\rho})\right)$ is in fact the square norm of the linear map $\partial_v\hat{\rho}$ with respect to the metric $H_i^-$.
	Using the decomposition of $W_i$ above, let's write $\partial_v\hat{\rho}$ as the decomposition $\partial_v\hat{\rho}=((\partial_v\hat{\rho})_1, (\partial_v\hat{\rho})_2)$ where $(\partial_v\hat{\rho})_1:\Im\hat{\rho}^*\to V_i$ and $(\partial_v\hat{\rho})_2:(\Im\hat{\rho}^*)^{\perp}\to V_i$. In particular, given the previous discussion, we see that $(\partial_v\hat{\rho})_1$ is actually $\partial_v\hat{\rho}$ composed with $\hat{\rho}^*(\hat{\rho}\hat{\rho}^*)^{-\frac{1}{2}}$.
	Thus, we can see that the other term  \[\tr\left(\left((\partial_v\hat{\rho})\cdot\left((\hat{\rho})^*(\hat{\rho}\hat{\rho}^*)^{-\frac{1}{2}}\right)\right)^*(\hat{\rho}\hat{\rho}^*)^{-1}\left((\partial_v\hat{\rho})\cdot\left((\hat{\rho})^*(\hat{\rho}\hat{\rho}^*)^{-\frac{1}{2}}\right)\right)\right)\] is actually the square norm of $\partial\hat{\rho}_1$ (with respect to the $H_i^-$ metric).
	Then we have $$i\partial\overline{\partial}\log\det H_i^-=|\partial\hat{\rho}|_{H_i^-}-|\partial\hat{\rho}_1|_{H_i^-}=|\partial\hat{\rho}_2|_{H_i^-}.$$
	Thus, the Ricci curvature is semi-positive definite.
	
	Now suppose $\left(\partial_v\hat{\rho}^{(i)}\right)_2=0$ for all $i$.  Then the image of $\left(\partial_v\hat{\rho}^{(i)}\right)^*=\partial_v(\hat{\rho}^{(i)})^*$ is in the image of $(\hat{\rho}^{(i)})^*$. Thus $\partial_v$ does not alter the subspaces given by $(\hat{\rho}^{(i)})^*:V_i\to W_i$. $((\hat{\rho}^{(i)})^*)_{i\in I}$ gives an embedding of $\cM^-$ to the product of Grassmannians of subspaces in $W_i$.   Since $\partial_v$ does not change the subspaces, it must be the zero tangent vector. As a result, the curvature is positive definite and defines a K\"ahler metric.
\end{proof}

\begin{example} \label{example:A2}
Consider the framed $A_2$ quiver. This quiver has vertices (1) and (2) and has one arrow $a$ going from $(1)$ to $(2)$. Then \[H_1^-=\left(Id_{d_1}-b_1b_1^*\right)^{-1}\] and \[H_2^-=\left(Id_{d_2}-b_2b_2^*-w_aw_a^*-w_ab_1b_1^*w_a^*\right)^{-1}=(Id_{d_2}-b_2b_2^*-w_aH^{-1}_1w_a^*)^{-1}\] where $H_1=(\Id_{d_1}+b_1b_1^*)^{-1}$ is the Hermitian metric on $\cV_1$ in Definition \ref{theorem:metric}.
%

Using Gram-Schmidt orthonormalization, we can write $H_1^{-1}=g(b_1)g(b_1)^*$ for some $g(b_1)\in \GL(d_1)$.  Thus \[H_2^-=(Id_{d_2}-(w_ag(b_1))(w_ag(b_1))^* - b_2b_2^*)^{-1}.\]

$\cM^- = \{(b_1,w_a,b_2): H_1^- \textrm{ and } H_2^- \textrm{ are positive definite}\}$.  Then we have the map 
$$\cM^- \to \Gr(n_1,d_1)^- \times \Gr(n_2+d_1,d_2)^- $$
by $(b_1,w_a,b_2) \mapsto (b_1,(w_ag(b_1),b_2))$, which is invertible.  We have identifications of the universal bundles $(\cV_i,H_i^-)$ with the pullback of tautological bundles over $\Gr(n_1,d_1)^-$ and $\Gr(n_2+d_1,d_2)^-$ respectively, which are compatible with this diffeomorphism.
\end{example}


We can go much further than this. In fact, for general acyclic quivers there exists an identification between $(\mathcal{V}_i^-,H_i^-)$ over $\cM^-$ and the tautological bundles over space-like Grassmannians as in the above example, if we ignore complex structures.


\begin{theorem}\label{thm:sympl}
Assume that the underlying quiver $Q$ is acyclic. Then there exists a symplectomorphism $$\phi:(\mathcal{M}^-,H_{T\mathcal{M}^-}^-)\stackrel{\cong}{\to}\prod\limits_i (\Gr^-(m_i, d_i),H_{(m_i, d_i)}^-)$$ that restricts to a diffeomorphism between the real loci, and a bundle isomorphism $$(\mathcal{V}_i^-, H_i^-)\overset{\cong}{\to}(\phi^*U_i, H_{(m_i, d_i)}^-)$$
that restricts to a bundle isomorphism between the corresponding real vector bundles over the real loci.  Here $U_i$ is the tautological bundle over $\Gr^-(m_i, d_i)$,
$m_i = n_i + \sum_{a: h(a)=i} d_{t(a)}$, and $H_{(m_i, d_i)}^-$ is the standard metric of $\Gr^-(m_i, d_i)$.

\end{theorem}

First, we make the following lemma. 

\begin{lemma}\label{lemma:metric}
$H_i^-=\left(\Id-b_ib_i^*-\sum\limits_{a:h(a)=i}w_aH_{t(a)}^{-1}w_a^*\right)^{-1}$.
\end{lemma}

\begin{proof}


Consider vertex $(i)$ in quiver $Q$. Let's denote $\Gamma_i:=\{\gamma:h(\gamma)=i\}$, the paths ending at $(i)$.

Aside from the trivial path, every $\gamma$ in $\Gamma_i$ must be of the form $a\cdot\gamma$ for some arrow $a$ with $h(a)=i$. Thus, we can decompose $\Gamma_i=\{(i)\}\cup \bigcup\limits_{a: h(a)=i}a\cdot\Gamma_{t(a)}$ where $(i)$ denotes the trivial path.  Because of this, we can write \[H_i^-=\left(\Id-b_ib_i^*-\sum_{a: h(a)=i}\sum_{\gamma\in\Gamma_{t(a)}}w_{a\cdot \gamma} e^{t(a\cdot\gamma)}\left(w_{a\cdot \gamma} e^{t(a\cdot\gamma)}\right)^*\right)^{-1}.\]
We have $w_{a\cdot\gamma}=w_{a}\cdot w_{\gamma}$ where $w_{a}$ is the linear map associated to the arrow $a$. 
Moreover,  $t(a\gamma)=t(\gamma)$.
Thus
\begin{align*}
	H_i^-=&\left(\Id-b_ib_i^*-\sum_{a: h(a)=i}\sum_{\gamma\in\Gamma_{t(a)}}w_{a}w_{\gamma} e^{t(\gamma)}\left(w_{a}w_{\gamma} e^{t(\gamma)}\right)^*\right)^{-1} \\
	=&\left(\Id-b_ib_i^*-\sum_{a: h(a)=i}w_{a}\left(\sum_{\gamma\in\Gamma_{t(a)}}w_{\gamma} e^{t(\gamma)}\left(w_{\gamma} e^{t(\gamma)}\right)^*\right)w_{a}^*\right)^{-1}\\
	=&\left(\Id-b_ib_i^*-\sum_{a: h(a)=i}w_{a}H_j^{-1}w_{a}^*\right)^{-1}.
\end{align*}
\end{proof}







\begin{proof}[Proof of Theorem \ref{thm:sympl}]




Let $(i)$ be a vertex. By Lemma \ref{lemma:metric}, we can write $H_i^-$ as \[\Id-b_ib_i^*-\sum_{a:h(a)=i}w_aH_{t(a)}^{-1}w_a^*.\]
By Gram-Schmidt normalization, we can write $H_{t(a)}^{-1}=g_{t(a)}g_{t(a)}^*$ for some $g_{t(a)}\in \GL(d_{t(a)})$.
Then
\[H_i^-=\left(\Id-b_ib_i^*-\sum_{a:h(a)=i}w_ag_{t(a)}g_{t(a)}^*w_a^*\right)^{-1}=\left(\Id-ww^*\right)^{-1}\]
where \[w=b_i\oplus\bigoplus\limits_{a:h(a)=i}w_ag_{t(a)}.\]
Thus, we define $\phi:(\mathcal{M}^-,H_{T\mathcal{M}^-}^-)\to\prod\limits_i (\Gr^-(m_i, d_i),H_{(m_i, d_i)}^-)$ by sending $(b_i,w_a)_{i \in Q_0, a \in Q_1}$ to $(b_i,(w_ag_{t(a)})_{a: h(a)=i})_{i \in Q_0}$.  $\phi$ is invertible: for each $i\in Q_0$, $g_i$ only depends on $b_j$ for $j \in Q^{(i)}_0$ and $w_a$ for $a \in Q^{(i)}_1$, where $Q^{(i)}$ is the sub-quiver containing those arrows that can be a part of a path heading to $i$.  Then we can solve back $w_a$ inductively from $b_i$ and $w_a g_{t(a)}$ (where $g_{t(a)}$ is invertible).  Since $\phi$ identifies $H_i^-$ with the standard metric on the tautological bundle of $\Gr^-(m_i,d_i)$, and the symplectic form is $H_T^-=-i\sum\limits_{i}\partial\overline{\partial}\log\det H_i^-$, $\phi$ is a symplectomorphism.  Written in these coordinates, $(\mathcal{V}_i^-, H_i^-)\overset{\cong}{\to}(\phi^*U_i, H_{(m_i, d_i)})$ is simply given by identity.

Restricting to $b_i$ and $w_a$ having real coordinates, $g_i$ produced from the Gram-Schmidt process is a real matrix.  Thus $\phi$ restricts as a diffeomorphism between the real loci.
\end{proof}

\begin{remark}
This correspondence between $\mathcal{M}^-$ and $\Gr^-(m_i, d_i),H_{(m_i, d_i)}^-$ is \textbf{only} a symplectomorphism, since the Gram-Schmidt process is not holomorphic.
\end{remark}

\subsection{Euclidean Signature}\label{section:Euclidean}

In addition to the non-compact dual, we can use Equation \ref{eq:uniform} to get other interesting moduli spaces in the same vein. The most straightforward variant is achieved by setting $\alpha$ and all of the $\alpha_{\gamma}$ to zero. This means throwing out the contribution coming from anything other than the first $d_i$-many framing arrows.

\begin{defn}

Assume $Q$ has no oriented cycle. Let $\rho_i$ be as in definition \ref{theorem:metric} so that $\rho_i$ is a row vector with entries of the form $w_{\gamma}e^{t(\gamma)}$ where $\gamma$ is some path in $Q$ ending at vertex $(i)$, $w_{\gamma}$ is the representing matrix of this path, and $e^{t(\gamma)}$ is the framing map at $t(\gamma)$, the starting vertex of $\gamma$. Arrange the entries of $\rho_i$ so that the first $n_i$-many entries correspond to the framing arrows at vertex $(i)$. Then let $H_i^0$ be the quadratic form defined by:

\begin{equation}
H_i^0=\left(\rho_i\left(\begin{matrix} I_{d_i} & 0 \\0 & 0\end{matrix}\right)\rho_i^*\right)^{-1}
\end{equation}

Here, $N_i=\sum\limits_{(j)\rightsquigarrow (i)}n_j$. We define $R_{n,d}^0$ to be the subset of $R_{n,d}$ where $H_i^0$ is positive-definite for all $i$.

\end{defn}

Note that we still need $N_i\geq d_i \forall i$ to have $\mathcal{M}\neq\emptyset$, thus we will still be assuming that to be the case. Indeed, most of the following statements are copied or follow from analogous statements in Section \ref{subsection:hyper}.

\begin{prop}\label{prop:R0}
$R_{n, d}^0\subset R_{n, d}^s$.
\end{prop}

\begin{proof}
As in Proposition \ref{prop:R-}, consider a point in $R_{n, d}^0$. Write $\rho_i=(\epsilon_i \,\, R)$ evaluated at this point as a $(d_i\times N_i)$-matrix, where $\epsilon_i$ is a $(d_i\times d_i)$-matrix and $R$ is the remaining part. Then $H_i^0=(\epsilon_i\epsilon_i^*)^{-1}$. If $\epsilon_i$ is not invertible, then $\epsilon_i\epsilon_i^*$ is not positive-definite. Thus, for a point in $R_{n, d}^0$, we have that $\epsilon_i$ is invertible for all $i$ which means that $\rho_i$ is surjective for all $i$. Thus the point is stable.
\end{proof}

\begin{lemma}
$H_i^0$ is $G_d$-equivariant and $R_{n, d}^0$ is $G_d$-invariant.
\end{lemma}

\begin{proof}

This follows directly from Lemma \ref{lemma:hypequiv}.

\end{proof}

Similar to Section \ref{subsection:hyper}, we will assume $n_i \geq d_i$ from this point forward. Thus, we can talk about the framing part $\epsilon_i$ of $e^{(i)}$ corresponding to the first $d_i$-many components, and the bias part $b_i$ of $e^{(i)}$ corresponding to the remaining $(n_i-d_i)$-many components. With this, $H_i^0$ can be written simply as \[H_i^0=(\epsilon_i\epsilon_i^*)^{-1}\]

\begin{prop}\label{prop:invertible-0}
Assume that $n_i\geq d_i$ for all $i$. \[\emptyset \not=R_{n,d}^0=\{\epsilon_i \text{ is invertible for all }i\} \subset R_{n,d}^s.\]

\end{prop}

\begin{proof}
From the proof of Proposition \ref{prop:R0}, it is clear that $\epsilon_i$ is invertible over $R_{n,d}^0$ and these points belong to $R_{n,d}^s$. Moreover, let $w$ be any point of $R_{n, d}^s$ such that the framing parts $\epsilon_i$ of the framing maps $e^{(i)}$ are all invertible. Since $H_i^0$ is only defined using $\epsilon_i$, we can see that $w\in R_{n, d}^0$. Thus, $R_{n, d}^0$ is the subset of $R_{n, d}^s$ of points where the framing part is invertible. To see that $R_{n,d}^0 \not= \emptyset$, we can take $\epsilon_i = \Id$ for all $i\in Q_0$ and set the remaining arrows to be zero. This gives a point in $R_{n,d}$ at which $H_i^0 = \Id$ is positive-definite.
\end{proof}


Similar to $R_{n, d}^-$, since we know that $R_{n,d}^0$ is a $\GL(d)$-invariant non-compact open subset of $R_{n, d}^s$, we can directly quotient by $\GL(d)$.

\begin{defn}
We define the Euclidean restriction of $\mathcal{M}$ as the quotient $\mathcal{M}^0=R_{n,d}^0/\GL(d)$ with universal bundles $\mathcal{V}_i^0:=(R_{n,d}^0\times \C^{d_i})/\GL(d_i)$. Since $H_i^0$ is Hermitian and $G_d$-equivariant, it descends to a metric on $\mathcal{V}_i^0$ over $\mathcal{M}^0$.
\end{defn}

As a result of Proposition \ref{prop:invertible-0} and the fact that $\GL(d)$ acts only on the left on the framing space, $e^{(i)}=(\epsilon_i, b_i)=(I_{d_i}, \tilde{b}_i)$ where $\tilde{b}_i=\epsilon_i^{-1}b_i$ and is itself a generic bias vector for each $i$. Thus, from this point forward we will be assuming both that $n_i\geq d_i$ for all $i$ and that all framing maps are of the form $e^{(i)}=(\Id, b_i)$. Thus $\mathcal{M}^0\cong R_{n-d,d}$ and $H_i^0$ can be taken to be the trivial metric on $\C^{d_i}$ for each $i$.

Over $\mathcal{M}^0$, activation functions have the simplest possible definition: smooth (or piece-wise smooth) maps from $\C^{d_i}$ to itself. Any of the standard activation functions used in machine learning (sigmoid, ReLu, softmax, etc.) directly fit in this Euclidean restriction setting without any further modification.



\begin{corollary}

$H_i^0$ is the trivial metric on $\C^{d_i}$. Thus, $H_T^0:=\sum\limits_i\partial\overline{\partial}\log\det H_i^0$ is a Ricci-flat K\"ahler-Einstein metric and $R_{n, d}^-\subset R_{n, d}^0$.
\end{corollary}








\begin{remark}

Consider an acyclic quiver $Q$ with dimension vector $(n, d)$ such that $n_i=d_i$ for all source and sink vertices $i$, and $n_i=d_i+1$ for all others. If $Q$ with dimension vector $(n, d)$ gives the underlying neuron structure for a neural network, then $\mathcal{M}^0$ is the training space for this network. In particular, the standard backward propagation algorithm for a feed-forward neural network is standard gradient descent in the relevant vector space, matching up exactly with the gradient descent on $\mathcal{M}^0$ induced by $H_i^0$.


\end{remark}

\subsection{Hyperbolic Activation Functions} \label{sec:hyp}
This point of view of uniformization provides a learning model over hyperbolic moduli, or more generally, interpolations of spherical, Euclidean and hyperbolic moduli.  (One can add learnable parameters in the Hermitian metrics $H_i$, interpolating the metrics of different types.)  This is hyperbolic learning in the base (that is the parameter space).  There is another direction that we can consider hyperbolic learning, namely the fiber bundle direction.

Recall that we have the universal vector bundles $\cV_i$.  In \cite{JL}, we cooked up activation function (as a fiber bundle map of $\cV_i$) by composing the following:
$$ \cV_i \stackrel{H_i}{\cong} \cV_i^* \stackrel{(e^{(i)})^*}{\to} \underline{\C^{n_i}} \stackrel{\sigma}{\to} \underline{\C^{n_i}} \stackrel{e^{(i)}}{\to} \cV_i $$
where $\sigma: \C^{n_i} \to \C^{n_i}$ is a continuous function.  We can do the same thing uniformly for $\cM, \cM^0$ and $\cM^-$.

Additionally, in \cite{JL}, we constructed a specific activation function as a symplectomorphism $(\C^n,\omega_{\bP^n}|_{\C^n}) \cong (B, \omega_{\textrm{std}})$, where $B \subset \C^n$ is the ball $\{\|\vec{z}\|^2 < 1\}$, $\omega_{\bP^n}$ is the Fubini-Study metric on $\bP^n$, and $\omega_{\textrm{std}}$ is the standard symplectic form of $\C^n$.  $\sigma$ has the expression
$$(z_1,\dots, z_n)\to\left(\frac{z_1}{\sqrt{1+\sum_{i=1}^n|z_i|^2}}, \dots, \frac{z_n}{\sqrt{1+\sum_{i=1}^n|z_i|^2}}\right).$$

In view of hyperbolic metrics, we provide an alternative interpretation of the same function here.

\begin{prop}
$\sigma$ gives a symplectomorphism
$(\C^n, \omega_{\textrm{std}})\to (\C\mathbb{H}^n, \omega_{\C\mathbb{H}^n})$
where $\C\mathbb{H}^n$ denotes the hyperbolic ball.
\end{prop}

\begin{proof}
By definition, $\omega_{\C\mathbb{H}^n}$ equals to $-\partial\overline{\partial}\log(1-|w|^2)$ up to a simple scaling. Here, we will be thinking of $w$ as the row vector $(w_1,\dots, w_n)$. Then



\[-\partial\overline{\partial}\log(1-|w|^2)=\partial\frac{wdw^*}{1-|w|^2}\]\[=\frac{(1-|w|^2)dw\wedge dw^*+(dw\cdot w^*)wdw^*}{(1-|w|^2)^2}=\frac{(1-|w|^2)dw\wedge dw^*+\overline{w}dw^t\overline{dw}w^t}{(1-|w|^2)^2}\]

Now, let's similarly write $z$ as the row vector $(z_1,\dots, z_n)$. We compute the pullback as

\[\sigma^*(dz\wedge d z^*)=d\frac{z}{\sqrt{1-|z|^2}}\wedge d\frac{z^*}{\sqrt{1-|z|^2}}\]
\[=\frac{(1- zz^*)dz+ \frac{1}{2}(\overline{zdz^*}+zdz^*)z}{(1- zz^*)^{3/2}}\wedge\frac{(1- zz^*)dz^*+ \frac{1}{2}(\overline{zdz^*}+zdz^*)z^*}{(1- zz^*)^{3/2}}\]
\[=\frac{1}{(1- zz^*)^3}\left((1- zz^*)^2dz\wedge dz^*+\frac{1}{4}(\overline{zdz^*}+zdz^*)^2zz^*\right)\]\[+\frac{1}{(1-zz^*)^3}\left(\frac{1}{2}(1-zz^*)\left((\overline{zdz^*}+zdz^*)z\wedge dz^*+dz\wedge(\overline{zdz^*}+zdz^*)z^*\right)\right)
\]

At this point, $\overline{zdz^*}\wedge zdz^*+ dz z^*\wedge zdz^*$ can be rewritten as $2 dz^t \overline{dz}z^t$. This gives us

\[=\frac{(1-zz^*)dz\wedge dz^*+\overline{z}dz^t\overline{d}zz^t}{(1-zz^*)^2}\]
which equals to above.
\end{proof}

In other words, $\sigma$ gives an identification between $\C^n$ and $\C\mathbb{H}^n$.  Then signal propagation between hyperbolic spaces can be modeled simply as linear maps between $\C^n$, and the composition $\iota \circ \sigma: \C^n \to \C^n$, where $\iota: \C\mathbb{H}^n \to \C^n$ is the inclusion of a ball in the space, gives an activation function.

\bibliographystyle{amsalpha}
\bibliography{geometry}	
\end{document}